\documentclass[11pt, reqno]{amsart}
\usepackage[T1]{fontenc}
\usepackage[utf8]{inputenc}
\usepackage{geometry}
\usepackage{newunicodechar}
\newunicodechar{̈}{\"{}}
\newunicodechar{̆}{\u{}}
\usepackage[backend=biber,style=numeric-comp,sorting=nyt, url=false,maxnames=5,minnames=1]{biblatex}
\defbibheading{bibliography}{\section*{References}} % Remove the heading of bibliography

\AtEveryBibitem{% Clean up the bibtex rather than editing it
 \clearlist{address}
 \clearfield{date}
 \clearfield{eprint}
 \clearfield{isbn}
 \clearfield{issn}
 \clearlist{location}
 \clearfield{month}
 \clearfield{series}
 \clearfield{doi}
 
 \ifentrytype{book}{}{% Remove publisher and editor except for books
  \clearlist{publisher}
  \clearname{editor}
 }
}

\usepackage{listings}
\usepackage{faktor}
\usepackage{color}
\usepackage{hyperref}
\usepackage{acro}
\usepackage{amsthm}
\usepackage{amssymb}
\usepackage{subcaption}
\usepackage{cancel}
\usepackage{tikz}
\usepackage{amsmath}
\usepackage{etoolbox}
\usepackage{bm}
\usepackage{tikz-cd}
\usepackage{pgfplots}
\pgfplotsset{compat=1.18}
\usepackage{mathtools}
\usepackage{dsfont}
\usepackage{graphicx}
\usetikzlibrary{arrows.meta}
\usepackage[T1]{fontenc} % Output font encoding for international characters

\newcommand{\R}{\mathbb{R}} %R insieme numeri reali
\newcommand{\normone}[1]{\left\|#1\right\|_{L^\infty(0,T;L^2(\Omega))}}
\newcommand{\normtwo}[1]{\left\|#1\right\|_{L^\infty(0,T;H^1(\Omega))}}
\newcommand{\normthree}[1]{\left\|#1\right\|_{L^\infty((0,T)\times\Omega)}}

\newtheorem{theorem}{Theorem}[section]
\newtheorem*{theorem*}{Theorem} 
\newtheorem{corollary}{Corollary}[section]
\newtheorem*{corollary*}{Corollary}
\newtheorem{proposition}{Proposition}[section]

\theoremstyle{definition}

\newtheorem*{definition*}{Definition}
\theoremstyle{remark}
\newtheorem{remark}{Remark}[section]

\numberwithin{equation}{section}
\newcommand{\bu}{{\mathbf u}}

\newcommand{\cC}{\mathcal{C}}
\usepackage{enumitem}
\begin{document}

\title[Vanishing viscosity limit  with anisotropic viscosity]{The Vanishing Viscosity Limit and Boundary Layers for Symmetric Fluid Flows with Anisotropic Viscosity}

\author[V. Galbiati]{Valentina Galbiati}
\email{vgalbiati@bcamath.org}
\address{BCAM, Basque Center for Applied Mathematics, Bilbao 48009, Spain}

\author[A. Mazzucato]{Anna Mazzucato}
\email{alm24@psu.edu}
\address{Department of Mathematics, Penn State University, University Park (PA) 16802, USA}

\author[R. Montalto]{Riccardo Montalto}
\email{riccardo.montalto@unimi.it}
\address{Department of Mathematics, Universit\' a degli Studi di Milano, Milano 20133, Italy}

\date{\today}

\begin{abstract}
    We study the vanishing viscosity limit for the incompressible Navier-Stokes equations with anisotropic viscosity in bounded domains, analyzing certain classes of symmetric flows: plane parallel, pipe parallel and circularly symmetric. By anisotropic viscosity, it is meant here that the viscosity coefficient in the direction normal to the wall is different than that in the direction tangential to the wall. Using boundary layer theory and semigroup techniques, we establish the validity of the vanishing viscosity limit in the energy norm, that is, in $L^2$ in space uniformly in time, for all three classes of flows,  with explicit convergence rates. We further obtain higher-order estimates under suitable assumptions on the anisotropic viscosity coefficients. In particular, we consider both the case in which the tangential viscosity coefficient goes to zero faster than the normal one and, conversely, the case when the normal coefficient vanishes faster then the tangential one. Our results extend previous works on isotropic viscosity and provide new examples where the vanishing viscosity limit can be rigorously justified in the anisotropic setting.
\end{abstract}



\maketitle

\section{Introduction} \label{s:intro}
%\addcontentsline{toc}{section}{Introduction}

This article concerns the limit of vanishing viscosity for incompressible Newtonian flows with anisotropic viscosity in bounded domains. 
By anisotropic viscosity we mean that near the boundary the viscosity coefficients depends on direction, and more specifically that the viscosity in directions tangent to the boundary is different than the viscosity in the direction normal to the boundary. Such a model is justified physically in several context. If the viscosity coefficient arises from the mean-free path of fluid particles, than it is reasonable to assume that the mean-free path is shorter in the direction normal to the wall, which justifies taking the normal viscosity smaller than the tangential one. Such a situation also arises in geophysical flows due to stratification (see e.g. \cite{CheminEtAll2000}). On the other hand,  a normal viscosity larger than the tangential one can arise if the flow and its container are subject to external forces, such as centrifugal forces due to strong rotation or the action of a strong magnetic field in ionized fluids.

In the isotropic case, the dynamics of viscous incompressible flow in a domain $\Omega\subset \R^d$, $d=2,3$, is described by the Navier–Stokes equations (NSE):
\begin{equation}
\label{NSeq}
\partial_t \mathbf{u} + (\mathbf{u}\cdot \nabla)\mathbf{u} - \nu \Delta \mathbf{u} + \nabla p = \mathbf{f},
\quad \operatorname{div}\mathbf{u} = 0,
\end{equation}
where $\mathbf{u}$ denotes the velocity field, $p$ the pressure, $\nu>0$ the viscosity coefficient, and $\mathbf{f}$ is an external forcing term. The system is complemented with the initial condition $\mathbf{u}^\nu(0)=\mathbf{u}^\nu_0$ and, in the case of impermeable walls, the classical no-slip boundary condition $\mathbf{u}|_{\partial \Omega}=0$, assuming that the fluid sticks to the boundary by friction.

When viscosity is small, the dynamic is formally approximated by the Euler equations for incompressible inviscid flows (EE):
\begin{equation}
\label{Eeq}
\partial_t \mathbf{u}^0 + (\mathbf{u}^0\cdot\nabla)\mathbf{u}^0 + \nabla p^0 = \mathbf{f},
\quad \operatorname{div}\mathbf{u}^0 = 0,
\end{equation}
with initial condition $\bu^0(0)=\bu_0$ and the so-called no-penetration condition $\mathbf{u}^0\cdot \mathbf{n}|_{\partial\Omega}=0$, where $\mathbf{n}$ is the outer unit normal to $\partial\Omega$. The different boundary conditions for Navier-Stokes and Euler are consistent with the different order of the equations, but lead to a potentially large discrepancy in the tangential components of the viscous and inviscid velocities.
It follows heuristically that the convergence $\mathbf{u} \to \mathbf{u}^0$ cannot hold uniformly up to the boundary. Instead, a thin region develops near $\partial \Omega$, the so-called viscous boundary layer, where viscosity cannot be neglected even as viscosity vanishes. Away from the boundary layer, one expects the viscous flow to be well approximated by inviscid flows at sufficiently low viscosity. 

We will refer to the vanishing viscosity limit  (VVL for short) as the strong convergence of Navier-Stokes solutions to Euler solutions in the energy norm, i.e., in $L^2(\Omega)$ uniformly in time:
\[
    \lim_{\nu\to 0} \|\bu^\nu-  \bu^0\|_{L^\infty(0,T;L^2(\Omega))},
\]
where $T$ is the time of existence of the Euler solutions. Given the current state of the well-posedness theory for Navier-Stokes and Euler, in the VVL $\bu^\nu$ can be taken to be a Leray-Hopf weak solution, while $\bu^0$ is regular solution.

The Navier-Stokes equations should be properly non-dimensionalized and the dimensional viscosity coefficient $\nu$ replaced by the adimensional Reynolds number $Re$, proportional to $1/\nu$, but with abuse of notation we will treat $\nu$ as an adimensional number.

Boundary layers are of both physical and mathematical importance. From a physical point of view, they are the regions where vorticity is generated and subsequently transported into the bulk of the fluid potentially causing turbulence and playing a crucial role in many applications. From a mathematical point of view, they give rise to a singular perturbation problem, the rigorous analysis of which remains one of the most challenging and open questions in fluid dynamics.

The typical mathematical approach to boundary layers is to construct an approximate solution correcting the formal limit solution (see e.g. \cite{VishikLjusternik1962}):
\begin{equation*}
    \mathbf{u}\approx \mathbf{u}^{app}=\mathbf{u^0}+\mathbf{\theta^0},
\end{equation*}
where the corrector $\mathbf{\theta^0}$ accounts for the mismatch at the boundary. In general, the form of the corrector is obtained from a formal asymptotic analysis after suitably rescaling variables so that the boundary layer becomes of size independent of $\nu$. Energy methods and corrector estimates are then employed to prove convergence, if it holds. 
A limitation of this method, however, is that it relies on a decomposition imposed {\em a priori}, and on assumptions about the behavior of the corrector at the boundary. A related, but different approach, approximates the Navier-Stokes solution as an outer solution, valid away from the boundary layer and typically the limit Euler solution, and an inner solution valid in the boundary layer \cite{VishikLjusternik1962}. After rescaling, the inner solution is found to have to satisfy the Prandtl equations , which however are known to be ill-posed without strong {\em a priori} hypotheses on the data, and hence of limited applicability in a rigorous study \cite{Grenier2000,GrenierGuoNguyen2016,GuoNguyen2011,GerardVaretDormy2010,EEnquist1997,KukavicaVicolWang2017,VanDommelenShen1980,GarganoSammmartinoSciacca2009} (we refer the reader also to the recent book \cite{QinDongWangBook2024}).

There is an extensive literature concerning boundary layers and Prandtl equations under various types of boundary conditions for incompressible flows (we refer the reader to the classic references \cite{LighthillBook1963,schlichting1961boundary,OleinikSamokhinBook1998} and the recent book \cite{gie2018singular}, along with survey articles such as \cite{BardosTiti2007,BardosTiti2013,inviscid,MaekawaSurvey2019,GieJungTemam2016}. The case of non characteristic boundaries, slip-with-friction or vorticity boundary conditions is better understood. Since these situations are not the focus of the article, we refer the reader to the discussion in 
\cite{BeiraodaVeigaCrispo2023} for friction conditions and in \cite{MazzucatoWangWei2026} for inflow/outflow conditions.

The case of no-slip boundary conditions remains the most difficult and is still far from being fully solved in general. Rigorous results are available only in special settings: in the analytic setting, using the Cauchy-Kowaleski Theorem \cite{Asano1988}, and more recently in suitable Gevrey classes
\cite{SammartinoCaflisch1998.a,SammartinoCaflisch1998.b,NguyenNguyen2018,GerardVaretEtAll2018,WangWangZhang2024,LiWang2020,PaicuZhang2021,IgnatovaVicol2016,ZhangZhang2016,WangWangZhang2017,KukavicaVicol2013,GerardVaretMasmoudi2015, Maekawa2014,NguyenNguyen2024,Maekawa2021,KukavicaVicolWang2022,BardosEtAll2022,Fei2020, CannoneLombardoSammartino2014}, under a monotonicity condition or other special structure \cite{AlexandreEtAl2015,Oleinik1963,GerardVaretEtAll2018,GuoNguyen2011,MasmoudiWong2018,QinWang2025,KukavicaEtAll2014,FeiTaoZhang2018},  for certain classes of flows with symmetry \cite{lopes2008vanishing,mazzucato2008vanishing,mazzucato2010vanishing,BonaWu2002,LopesMazzucatoNussenzveigLopes2008,plane,GIE20191237,han2012boundary}, and for linearized flows \cite{2018JMFM...20.1405G,Gie2014,TemamWang1995,TemamWang1996,TemamWang1998,LombardoSammartino2001}. Convergence criteria, such as the well-known Kato criterion \cite{Wang2001}, also provide conditions equivalent to the validity of the limit, but they are not known to hold for generic flows, not even for short time (we mention, among the extensive literature, \cite{Kelliher2007,kelliher2008vanishing,kelliher2017observations,Kelliher2023,ConstantinEtAl2019,ConstantinKukavicaVicol2015,ConstantinVicol2018,ConstantinEtAll2017,BardosTiti2007,BardosTitiWiedemann2019,BardosTiti2013,DrivasNguyen2018,DrivasNguyen2019,SeisEtAll2026,ChenEtAl2022,Wang2001,KukavicaVicolWang2022}). There are also rigorous bounds on the possible rate of separation between Euler and Navier-Stokes \cite{VasseurYang2023,VasseurYang2024}.

In the context of isotropic symmetric flows, semigroup theory has been applied to establish the vanishing viscosity limit \cite{lopes2008vanishing,mazzucato2008vanishing,mazzucato2010vanishing}. This approach is particularly interesting since it avoids the introduction of correctors and, in general, requires weaker assumptions on the initial data. By symmetric flows, we mean here exact solutions of the Navier-Stokes and Euler equations that preserve a certain symmetry imposed on the initial data. The vanishing viscosity limit has been established for so-called pipe and channel flows in periodic pipes and channels. These are three-dimensional flows, for which however the dynamic of the flow is captured by a reduced two-by-two linear system for the velocity weakly coupled non-linearly to the other component of the velocity. The symmetry of the flow maintains the flow laminar and prevents layer separation. Pipe flows reduce to circularly symmetric flows in the planar case. They can be viewed as generalizations of the classical Couette and Poiseulle flows and are of numerical and, even, experimental interest (cf. turbulent channel and pipe flow simulations, see e.g. \cite{LeeMoser2015} and references therein).

In this article, we study the vanishing viscosity limit and the boundary layer in the same classes of symmetric flows, but for anisotropic flows. In this framework, the usual isotropic viscous term $\nu \Delta$ is replaced by a more general elliptic operator $L_{\nu}$, where the viscosity depends on the direction. In the anisotropic case, the validity of the vanishing viscosity limit may depend crucially on the ratio between the tangential and the normal viscosity coefficients at the boundary.

There is a significantly less developed literature concerning the vanishing viscosity limit for anisotropic viscosity with no-slip boundary conditions. The first general result available in the literature is due to Masmoudi \cite{masmoudi1998euler} (1998), who proved that the vanishing viscosity limit holds for the anisotropic system if also the ratio of normal to tangential viscosity tends to zero in the channel or in the half space. A successive work of Phan and Valdebenito \cite{nonflat} (2022) shows the validity of the limit in a curved domain with the same constraint on the coefficients (see also the recent preprint \cite{GoodairArXiV2026} in the stochastic setting).  Other works have considered instead the so-called vanishing vertical viscosity limit, in which the horizontal viscosity is kept fixed while only the vertical component vanishes \cite{TAO20184283, cao2025vanishing, an3d}.

The purpose of this work is to give examples of flows where the vanishing viscosity limit and the boundary layer behavior can be rigorously proved under general conditions on the anisotropy of the viscosity, in particular in the case in which the tangential viscosity is much smaller than the normal viscosity.   In the isotropic setting, the cases with symmetries have been analyzed in literature using different methods, such as correctors and energy estimates or semigroup analysis. Correctors allow to prove convergence in higher norm for well-prepared data. In the anisotropic setting, we will establish  the vanishing viscosity limit in the energy norm $L^\infty L^2$ without assuming any condition on the tangential and normal viscosity. We will also provide an analysis of the convergence in higher norm, imposing some constraints on the viscosity coefficients. 

%inserire struttura articolo
 The article is organized as follows. In Section \ref{s:plane} we prove the VVL in the plane parallel case using correctors, assuming suitable conditions on the ratio of normal to tangential viscosity. We also establish bounds on the $L^\infty H^1$, $L^\infty L^\infty$ norm and discuss the vorticity equation. In Section \ref{s:pipe} we prove the VVL in the pipe parallel case, again using correctors, discussing in particular the modifications required by the curved boundary. In Section \ref{s:semigroup} we use semigroup theory to show the validity of the limit in the circularly symmetric case with no assumptions on the relative strength of normal versus tangential viscosity.

\section{Plane Parallel Channel Flow} \label{s:plane}
In this section, we study the vanishing viscosity limit under so-called plane-parallel symmetry in an infinitely long channel, which we take to be horizontal and of unit width without loss of generality, with periodicity in the horizontal directions.
We therefore pose the equations in the  spatial domain $\cC=[0,L]\times[0,L]\times[0,1]$, where $L$ is the horizontal period.
The velocity is assume to have the form (see figure \ref{fig:channel-boundary-layer}):
\begin{equation} \label{eq:planeSymmetry}
    \mathbf{u}(t,x,y,z)=(u_1(t,z), u_2(t,x,z),0),
\end{equation}
which is automatically divergence free and satisfying the no penetration condition at the channel walls. We impose that the initial condition and the forcing term have the compatible symmetry:
\begin{equation*}
    \mathbf{u}|_{t=0}=\mathbf{u}_0(x,y,z)=(a(z),b(x,z),0),
\end{equation*}
\begin{equation*}
    \mathbf{f}=(f_1(t,z), f_2(t,x,z),0).
\end{equation*}
Then solutions of the Navier-Stokes and Euler equations  necessarily satisfy \eqref{eq:planeSymmetry}, at least when sufficiently regular as it will be the case here.
 
\begin{figure}[ht]
\centering
\begin{tikzpicture}[scale=3.5]

  % Canale
  \draw[line width=2] (0,0) rectangle (3,1);

  % Linee tratteggiate per boundary layer
  %\draw[dashed, thick] (0,0.1) -- (3,0.1);
  %\draw[dashed, thick] (0,0.9) -- (3,0.9);

  % Frecce nere grosse (linee di flusso)
  \foreach \z in {0.25, 0.5, 0.75} {
    \foreach \x in {0.2, 0.5, 0.8, 1.1, 1.4, 1.7, 2.0, 2.3, 2.6} {
      \draw[->, -{latex}, line width=0.25pt] (\x,\z) -- ++(0.2,0);
    }
  }

  % Testo al centro
  \node at (1.4,0.62) {\footnotesize $\mathbf{u} = \left(u_1(t;z), u_2(t;x,z), 0\right)$};

  % Etichette coordinate
  \node[left] at (0,1) {\footnotesize $z = 1$};
  \node[left] at (0,0) {\footnotesize $z = 0$};
  \node[below right] at (2.5,0) {\footnotesize $(x,y) \in [0,L]^2$};

  % Etichette boundary layers con frecce
  %\draw[->, thick] (1.8,0.88) -- (1.8,1.06);
  %\node[above] at (1.8,1.05) {\footnotesize Upper boundary layer (near $z=1$)};

  %\draw[->, thick] (1.6,0.12) -- (1.6,-0.15);
  %\node[below] at (1.6,-0.15) {\footnotesize Lower boundary layer (near $z=0$)};

\end{tikzpicture}
\caption{Parallel flow in a channel}
\label{fig:channel-boundary-layer}
\end{figure}
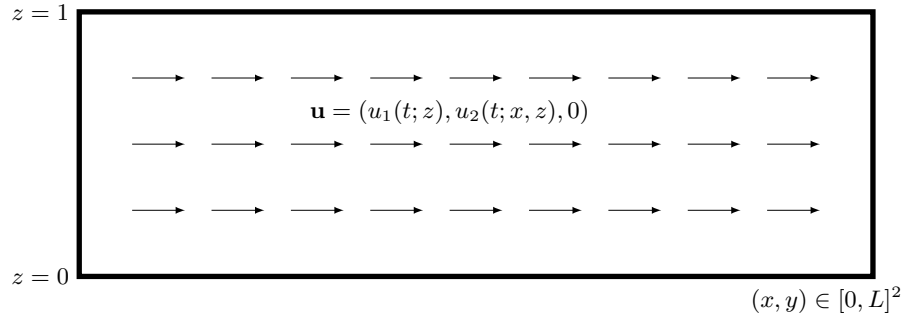

We consider the case of anisotropic viscosity. Hence, the scalar viscosity coefficient is replaced by a matrix-values coefficient in the dissipative term in the Navier-Stokes equations. That is, the term  $\nu\Delta$ in \eqref{NSeq} is replaced by the operator\ $\nu_1\partial_{xx}+\nu_2\partial_{yy}+\nu_3\partial_{zz}$, where $\nu_1,\nu_2,\nu_3$ are the viscosity respectively along the $x,y,z$ directions. 
Imposing plane parallel symmetry yields the following system of equations:
\begin{equation}
\label{NSEPP}
    \begin{cases}
        \partial_tu_1^\nu+\partial_xp^\nu-\nu_3\partial_{zz}u^\nu_1=f_1  & \text{ on } \cC\times(0,T)\\
        \partial_tu^\nu_2+u_1^\nu\partial_xu_2^\nu+\partial_yp^\nu-\nu_1\partial_{xx}u_2^\nu-\nu_3\partial_{zz}u_2^\nu=f_2 & \text{ on } \cC\times(0,T),\\
        \partial_zp^\nu=0 &\text{ on } \cC\times(0,T), \\
        u_1^\nu=u_2^\nu=0 & \text{ at } \partial \cC\times(0,T),\\
        u_1^\nu|_{t=0}=a(z) & \text{ on }  \cC,\\
        u_2^\nu|_{t=0}=b(x,z) & \text{ on }  \cC,  
    \end{cases}
\end{equation}
where all the unknowns are periodic in the horizontal directions.

We observe that the pressure $p^\nu$ may be chosen to vanish identically. Indeed, by differentiating the first equation respects $x$ and the second respect $y$, we obtain that $\partial_{xx}p^\nu=\partial_{yy}p^\nu=0$. As we are assuming periodic boundary conditions, we find that $p^\nu=p^\nu(z,t)$, but also $\partial_zp^\nu=0$, thus we can assume $p^\nu=0$. 
Since the system (\ref{NSEPP}) is independent to $\nu_2$,  the analysis of the vanishing viscosity limit  and the boundary layer is independent of the viscosity coefficient in the $y$ direction. 
Formally taking the limit for $\nu_1,\nu_2,\nu_3 \to 0$ gives the system:
\begin{equation}
\label{EEPP}
    \begin{cases}
        \partial_tu_1^0=f_1  &  \text{ on } \cC\times(0,T),\\
        \partial_tu^0_2+u_1^0\partial_xu_2^0=f_2 &  \text{ on } \cC\times(0,T),\\
        u_1^0|_{t=0}=a(z) & \text{ on } \cC,\\
        u_2^0|_{t=0}=b(x,z) & \text{ on } \cC.
    \end{cases}
\end{equation}
We impose also the compatibility conditions:
\begin{equation}
    \begin{aligned}
    \label{CC1}
        \mathbf{u}_0(x,y,i)&=0, \quad \\
        \nu_3\partial_{zz}a(i)+f_1&(0,i)=0,\\
        a(i)\partial_xb(x,i)=\nu_1\partial_{xx}b(x,i)+&\nu_3\partial_{zz}b(x,i)+f_2(0,x,i),
    \end{aligned}
\end{equation} 
for $i=0,1$ that prevent the formation of an initial layer in the NSE evolution due to the difference in boundary values between the initial data and the fluid velocity for positive time. The initial layer can be treated by introducing further correctors (see e.g. the discussion in \cite{gie2018singular}).

 From now on, since the reduced system is independent from $y$, we pose the problem in $\Omega=[0,L]\times[0,1]$ and we set:
\begin{equation*}
\begin{aligned}
    \mathbf{u}^\nu(t,x,z)&:=(u_1(t,z), u_2(t,x,z)), \\
    \mathbf{u}^0(t,x,z)&:=(u_1^0(t,z), u_2^0(t,x,z)), \\
    \mathbf{u}_0(t,z,x)&:=(a(z),b(x,z)),\\
    \mathbf{f}(t,x,z)&:=(f_1(t,z), f_2(t,x,z)),
\end{aligned}
\end{equation*}
where $\mathbf{u}^\nu$ is the solution of the system (\ref{NSEPP}), $\mathbf{u}^0$ of (\ref{EEPP}) and $\mathbf{u}_0$ is the initial data. We stress that $\mathbf{u}^\nu=\mathbf{u}^{\nu_1,\nu_3}$, but we don't make explicit this dependence below. All  functions are assumed periodic in the $x$ variable and boundary conditions will be given only at $z=0$ and $z=1$. 
%We observe that all the results will be shown in $\Omega$, but by the independence of the system from the $y$ coordinate we can extend all of them in the physical domain $Q$.

 Thanks to the weak coupling in the systems (\ref{NSEPP}) and (\ref{EEPP}), well-posedness is established using standard methods, similarly to the isotropic case (see e.g. \cite{mazzucato2010vanishing} and references therein). %as in \cite{evans2022partial}.
%We observe that (\ref{EEPP}) consists of an ordinary differential equation and a  transport equation, so the solution is regular if the initial data are regular enough. If we assume 
Assuming $\mathbf{u}_0\in H^m(\Omega)$, $m>5$, $\mathbf{f}\in C^\infty([0,T]\times\Omega)$ for simplicity, and the compatibility conditions, one has that $\mathbf{u}^0 \in C^0(0,T;H^m(\Omega)) \cap C^1(0,T;H^{m - 1}(\Omega)) $ and $\mathbf{u}^\nu \in C^0(0,T;H^m(\Omega)) \cap C^1(0,T;H^{m - 2}(\Omega)) $ and hence $\mathbf{u}^0, \mathbf{u}^\nu$ are classical solutions. %\Annacomment{both solutions are regular, I do not see why the viscous solution should be less regular}
%then  the solution $\mathbf{u}^0 \in C^1(0,T;H^m(\Omega))$, hence a classical solution of the problem. Concerning $\mathbf{u}^\nu$, using results in \cite{TEMAM198273} we have that: if $\mathbf{u}_0\in H^m(\Omega)$, $m>5$, $\mathbf{f} \in C^\infty([0,T]\times\Omega)$ and the compatibility conditions (\ref{CC1}) hold then $\mathbf{u}^\nu \in C(0,T;H^m(\Omega)\cap H^1_0(\Omega))$.

 Due to the mismatch in the boundary conditions of $\mathbf{u}^\nu$ and $\mathbf{u}^0$, a strong boundary layer forms that we proceed to analyze using correctors.
 
\subsection{Boundary Layer Analysis} \label{ss:PlaneBoundaryLayer}
Following the approach in \cite{VishikLjusternik1962}, we define an approximate solution $\bu^{app}$ of the viscous problem as follows: 
\begin{equation}
\label{APP}
    \mathbf{u}^{app}(t,x,z):=\mathbf{u}^{ou}(t,x,z)+\bm{\theta}^0\left(t,x,\frac{z}{\sqrt{\nu_3}}\right)+\bm{\theta}^{u,0}\left(t,x,\frac{1-z}{\sqrt{\nu_3}}\right),
\end{equation}
where $\mathbf{u}^{ou}$ is the {\em outer solution}, which is expected to represent the fluid velocity outside the boundary layer and $\bm\theta^0$, $\bm\theta^{u,0}$ are the {\em correctors} respectively near the lower ($z=0$) and the upper ($z=1$) wall of the channel. The role of the correctors is primarily to compensate for the difference between the tangential components of the viscous and inviscid velocities near the boundaries. 
A formal scaling analysis suggests the following form for the correctors:
\begin{equation*}
\begin{aligned}
    \bm\theta^0\left(t,x,\frac{z}{\sqrt{\nu_3}}\right)&:=\left(\theta_1^0\left(t,\frac{z}{\sqrt{\nu_3}}\right),\theta_2^0\left(t,x,\frac{z}{\sqrt{\nu_3}}\right)\right),\\
        \bm\theta^{u,0}\left(t,x,\frac{1-z}{\sqrt{\nu_3}}\right)&:=\left(\theta_1^{u,0}\left(t,\frac{1-z}{\sqrt{\nu_3}}\right),\theta_2^{u,0}\left(t,x,\frac{1-z}{\sqrt{\nu_3}}\right)\right),
        \end{aligned}
\end{equation*}
with $\bm\theta^0$, $\bm\theta^{u,0}$ independent of $\nu$. 
We therefore introduce the {\em stretched} variables $Z=\frac{z}{\sqrt{\nu_3}}$ and $Z^u=\frac{1-z}{\sqrt{\nu_3}}$ and we define $\Omega_{\infty}:=[0,L]\times[0,\infty)$ as the spatial domain for the correctors $\bm\theta^0(t,x,Z)$ and $\bm\theta^{u,0}(t,x,Z^u)$. 

For ease of notation, throughout the article we avoid indicating explicitly the dependence on $\nu$ except for the Navier-Stokes solution.

 We assume that the correctors satisfy the following matching conditions:
\begin{equation*}
    \theta_i^0\to0 \quad \textit{as} \quad Z\to\infty, \quad \theta_i^{u,0}\to0 \quad \textit{as} \quad Z^u\to\infty, \quad i=1,2.
\end{equation*}
These conditions encode the assumption that the approximate solution should be close to the outer solution outside of the layer.

We insert (\ref{APP}) into (\ref{NSEPP}) to derive equations for the correctors:
\begin{equation*}
\begin{aligned}
  (a)\quad&\partial_tu_1^{ou}+\partial_t\theta_1^0+\partial_t\theta_1^{u,0}-\nu_3\partial_{zz}u_1^{ou}-\partial_{ZZ}\theta_1^0-\partial_{Z^uZ^u}\theta_1^{u,0}=f_1,\\     (b)\quad&\partial_tu_2^{ou}+u_1^{ou}\partial_xu_2^{ou}+u_1^{ou}\partial_x\theta_2^0+u_1^{ou}\partial_x\theta_2^{u,0}+\partial_t\theta_2^0+\partial_t\theta_2^{u,0}+\theta_1^0\partial_xu_2^{ou}\\       &+\theta_1^0\partial_x\theta_2^0+\theta_1^0\partial_x\theta_2^{u,0}+\theta_1^{u,0}\partial_xu_2^{ou}+\theta_1^{u,0}\partial_x\theta_2^0+\theta_1^{u,0}\partial_x\theta_2^{u,0}-\nu_1\partial_{xx}u_2^{ou}\\
        &-\nu_1\partial_{xx}\theta_2^0-\nu_1\partial_{xx}\theta_2^{u,0}-\nu_3\partial_{zz}u_2^{ou}-\partial_{ZZ}\theta_2^0-\partial_{Z^uZ^u}\theta_2^{u,0}=f_2.
        \end{aligned}
\end{equation*}
We now take the outer solution to agree with the Euler solution, that is, we assume that $\bu^{ou}$ satisfies the system: 
    \begin{equation*}
    (1)\begin{cases}
        \partial_tu_1^{ou}=f_1,\\
        \partial_tu_2^{ou}+u_1^{ou}\partial_xu_2^{ou}=f_2.
        \end{cases}
    \end{equation*}
 with initial data $\bu^{ou}(0)=\bu_0$. By uniqueness, $\bu^{ou}=\bu^0$. Next, we observe that due to the regularity of the Euler solution, by Taylor expanding near the walls we can replace $\bu^0$ by its trace at the wall, as long as it is not differentiated in the direction normal to the walls. Then dropping the lower-order terms in $\nu_1$ and $\nu_3$  gives:
    \begin{equation}
    (2)\begin{cases}
        \partial_t\theta_1^0-\partial_{ZZ}\theta_1^0=0,\\
\partial_t\theta_2^0+\theta_1^0\partial_xu_2^0(t,x,0)+\theta_1^0\partial_x\theta_2^0+u_1^0(t,0)\partial_x\theta_2^0-\partial_{ZZ}\theta_2^0=0, \quad \quad \quad \quad \quad\\
        (\theta_1^0,\theta_2^0)|_{Z=0}=(-u_1^0(t,0),-u_2^0(t,x,0)),\\
        (\theta_1^0,\theta_2^0)|_{Z=\infty}=(0,0),\\
        (\theta_1^0,\theta_2^0)|_{t=0}=(0,0).
        \end{cases}
    \end{equation}
    \begin{equation}
    (3)\begin{cases}
        \partial_t\theta_1^{u,0}-\partial_{Z^uZ^u}\theta_1^{u,0}=0,\\
\partial_t\theta_2^{u,0}+\theta_1^{u,0}\partial_xu_2^0(t,x,1)+\theta_1^{u,0}\partial_x\theta_2^{u,0}+u_1^0(t,1)\partial_x\theta_2^{u,0}-\partial_{Z^uZ^u}\theta_2^{u,0}=0,\\
        (\theta_1^{u,0},\theta_2^{u,0})|_{Z^u=0}=(-u_1^{0}(t,1),-u_2^{0}(t,x,1)),\\
        (\theta_1^{u,0},\theta_2^{u,0})|_{Z^u=\infty}=(0,0),\\
        (\theta_1^{u,0},\theta_2^{u,0})|_{t=0}=(0,0),
        \end{cases}
    \end{equation}
where the boundary conditions are chosen to ensure that the approximate solution satisfies the no-slip condition.
These systems are well posed, even though degenerate parabolic. By energy estimates, it is possible to establish rates of decay for the correctors (see Appendix \ref{s:rates}).

In order to have the approximate solution supported on the physical domain $\Omega$, it is necessary to introduce a cut-off function to truncate the correctors.
Let $\psi(z)$ be a smooth function defined on $[0,1]$ such that  $\psi(z)=1$ when $z \in [0,\frac{1}{3}]$, and $\psi(z)=0$ when $z\in[\frac{1}{2},1]$. We then define the approximate solution in the physical domain $\Omega$ as \ $\tilde{\mathbf{u}}^{app}(t,x,z)=(\tilde{u}_1^{app}(t,z),\tilde{u}_2^{app}(t,x,z))$, where
\begin{equation*}
\begin{aligned}
    &\tilde{u}_1^{app}(t,z):=u_1^0(t,z)+\psi(z)\theta_1^0\left(t,\frac{z}{\sqrt{\nu_3}}\right)+\psi(1-z)\theta_1^{u,0}\left(t,\frac{1-z}{\sqrt{\nu_3}}\right),\\
    &\tilde{u}_2^{app}(t,x,z):=u_2^0(t,x,z)+\psi(z)\theta_2^0\left(t,x,\frac{z}{\sqrt{\nu_3}}\right)+\psi(1-z)\theta_2^{u,0}\left(t,x,\frac{1-z}{\sqrt{\nu_3}}\right).
    \end{aligned}
\end{equation*}

 It follows that $\tilde{\mathbf{u}}^{app}$ satisfies the following equations:
\begin{equation*}
\label{EQUAPP1}
    \partial_t\tilde{u}_1^{app}-\nu_3\partial_{zz}\tilde{u}_1^{app}=f_1+A+B,
\end{equation*}
\begin{equation*}
\label{EQUAPP2}
   \partial_t\tilde{u}_2^{app}+\tilde{u}_1^{app}\partial_x\tilde{u}_2^{app}-\nu_1\partial_{xx}\tilde{u}_2^{app}-\nu_3\partial_{zz}\tilde{u}_2^{app}=f_2+C+D+E;
\end{equation*}
with:
\begin{equation*}
\begin{aligned}
    &A=-2\sqrt{\nu_3}(\psi'(z)\partial_Z\theta_1^0+\psi'(1-z)\partial_{Z^u}\theta_1^{u,0}),\\
    &B=-\nu_3(\partial_{zz}u_1^0+\psi''(z)\theta_1^0+\psi''(1-z)\theta_1^{u,0}),\\
    &C=\psi(z)(\psi(z)-1)\theta_1^0\partial_x\theta_2^0+\psi(1-z)(\psi(1-z)-1)\theta_1^{u,0}\partial_x\theta_2^{u,0}-\nu_1\partial_{xx}u_2^0\\
    &-\nu_1\psi(z)\partial_{xx}\theta_2^0
-\nu_1\psi(1-z)\partial_{xx}\theta_2^{u,0}
+\psi(z)(u_1^0-u_1^0(t,0))\partial_x\theta_2^0+\psi(z)\theta_1^0(\partial_xu_2^0\\
& -\partial_x u_2^0(t,x,0))+\psi(1-z)(u_1^0-u_1^0(t,1))\partial_x\theta_2^{u,0} +\psi(1-z)\theta_1^{u,0}(\partial_x u_2^0-\partial_xu_2^0(t,x,1)),\\
    &D=-2\sqrt{\nu_3}(\psi'(z)\partial_Z\theta_2^0+\psi'(1-z)\partial_{Z^u}\theta_2^{u,0}),\\
    &E=-2\nu_3(\psi'(z)\partial_z\theta_2^0+\psi'(1-z)\partial_{Z^u}\theta_2^{u,0}).
    \end{aligned}
\end{equation*}
Above $\psi'$ indicates the derivative with respect to the argument.
This system is complemented by initial and boundary conditions.
\begin{equation*}
    \tilde{\mathbf{u}}^{app}|_{t=0}=\mathbf{u}_0,
\end{equation*}
\begin{equation*}
    \tilde{\mathbf{u}}^{app}|_{z=0}= \tilde{\mathbf{u}}^{app}|_{z=1}=0.
\end{equation*}

 We define the approximation error $\mathbf{u}^{err}(t,x,z):=\mathbf{u}^\nu(t,x,z)-\tilde{\mathbf{u}}^{app}(t,x,z)$, which satisfies the following system:
\begin{equation}
\label{ERR1}
    \partial_tu_1^{err}-\nu_3\partial_{zz}u_1^{err}=-(A+B),
\end{equation}
\begin{equation}
    \label{ERR2}
    \partial_tu_2^{err}+u_1^{err}\partial_x\tilde{u}_2^{app}+u_1^\nu\partial_xu_2^{err}-\nu_1\partial_{xx}u_2^{err}-\nu_3\partial_{zz}u_2^{err}=-(C+D+E);
\end{equation}
with
\begin{equation*}
    \mathbf{u}^{err}|_{t=0}=0,
\end{equation*}
\begin{equation*}
    \mathbf{u}^{err}|_{z=0}=\mathbf{u}^{err}|_{z=1}=0.
\end{equation*}
The task is then to obtain bounds on the error in terms of the viscosity coefficients. 

\subsection{Convergence rates} \label{ss:PlaneConvergence}

Our main results for anisotropic plane parallel flows is the following theorem

\begin{theorem}
\label{teo1}
    Let $\mathbf{ u}_0 \in H^m(\Omega), m>5$, $\mathbf{f} \in C^\infty([0,T]\times\Omega)$ and assume the compatibility condition (\ref{CC1}). Then there exist positive constants $C_i, i=1,2,3,4,5,$ independent of $\nu_1, \nu_3$ such that:
\begin{equation*}
    \normone{\mathbf{u}^\nu-\tilde{\mathbf{u}}^{app}}\leq C_1(\nu_3^\frac{3}{4}+\nu_1).
\end{equation*}
If $\nu_1 \leq \nu_3$:
\begin{equation*}
    \normtwo{\mathbf{u}^\nu-\tilde{\mathbf{u}}^{app}}\leq C_2(\nu_3^\frac{1}{4}+\nu_1^\frac{1}{2})\leq 2C_2\nu_3^\frac{1}{4},
\end{equation*}
\begin{equation*}
    \normthree{\mathbf{u}^\nu-\tilde{\mathbf{u}}^{app}}\leq C_3(\nu_3^\frac{1}{2}+\nu_1^\frac{3}{4}+\nu_3^\frac{3}{8}\nu_1^\frac{1}{4}+\nu_1^\frac{1}{2}\nu_3^\frac{1}{8})\leq 2C_3\nu_3^\frac{1}{2}.
    \end{equation*}
If $\nu_3 < \nu_1 < \nu_3^\frac{1}{2\alpha}$ with $\frac{1}{2}<\alpha<1$:
\begin{equation*}
    \normtwo{\mathbf{u}^\nu-\tilde{\mathbf{u}}^{app}}\leq C_4(\nu_3^\frac{1}{4}+\nu_1^{1-\alpha})\leq 2C_4\nu_1^\frac{1}{4},
\end{equation*}
\begin{equation*}
    \normthree{\mathbf{u}^\nu-\tilde{\mathbf{u}}^{app}}\leq C_5(\nu_3^\frac{1}{2}+\nu_1^\frac{2-\alpha}{2}+\nu_3^\frac{3}{8}\nu_1^\frac{1-\alpha}{2}+\nu_1^\frac{1}{2}\nu_3^\frac{1}{8})\leq 2C_5\nu_1^\frac{1}{2}.
    \end{equation*}
\end{theorem}

\begin{corollary}
Under the  hypotheses of Theorem \ref{teo1}, the following rate of convergence holds:
\begin{equation*}
    \normone{\mathbf{u}^\nu-\mathbf{u}^0}\leq C(\nu_3^\frac{1}{4}+\nu_1)
\end{equation*}
where $C$ is a constant depending of $\mathbf{u}_0$ but independent of $\nu_1, \nu_3$.
\end{corollary}

 This corollary is an immediate consequence of the following estimates for the correctors:
 $$\normone{\bm \theta^0}, \normone{\bm \theta^{u,0}} \leq c\nu_3^\frac{1}{4},$$ established below, and the triangle inequality. We observe that in this case the rate of convergence depends also from $\nu_1$, the tangential viscosity coefficient.

\begin{proof}[Proof of Theorem \ref{teo1}]  We start by performing energy estimates on $\mathbf{u}^{err}$.
By multiplying (\ref{ERR1}) by $u_1^{err}$ and integrating over $\Omega$ we obtain: 
\begin{equation*}
\begin{aligned}
    &\frac{1}{2}\frac{d}{dt}||u_1^{err}||_{L^2(0,1)}^2+\nu_3||\partial_zu_1^{err}||_{L^2(0,1)}^2=-\int_0^1(A+B)u_1^{err}dz\\
%    &=\int_{\frac{1}{3}}^{\frac{2}{3}}[2\sqrt{\nu_3}(\psi'(z)\partial_Z\theta_1^0+\psi'(1-z)\partial_{Z^u}\theta_1^{u,0})+\nu_3(\partial_{zz}u_1^0+\psi''(z)\theta_1^0+\psi''(1-z)\theta_1^{u,0})]u_1^{err}dz\\
    &\leq2\int_{\frac{1}{3}}^{\frac{2}{3}}|\sqrt{\nu_3}(\partial_Z\theta_1^0+\partial_{Z^u}\theta_1^{u,0})+\nu_3(\theta_1^0+\theta_1^{u,0})||u_1^{err}|dz+\nu_3||u_1^0||_{H^2(0,1)}||u_1^{err}||_{L^2(0,1)}\\
    &\leq 2\nu_3||u_1^{err}||_{L^2(0,1)}(||u_1^0||_{H^2(0,1)}+||\theta_1^0||_{L^2(0,\infty)}+||\theta_1^{u,0}||_{L^2(0,\infty)})+2\int_{\frac{1}{3}}^{\frac{2}{3}}|\sqrt{\nu_3}(\partial_Z\theta_1^0+\partial_{Z^u}\theta_1^{u,0})||u_1^{err}|dz.
\end{aligned}
\end{equation*}
%Our intent now is to find a best approximation of the last term. This is potentially possible because $\theta_1^0$ and $\theta_1^{u,0}$ depend on the thickness of the boundary layer.
The first term in the integral on the right-hand side can be estimated by
\begin{equation*}
\begin{aligned}
    &\int_{\frac{1}{3}}^{\frac{2}{3}}\left|\partial_Z\theta_1^0\left(t,\frac{z}{\sqrt{\nu_3}}\right)\right||u_1^{err}|dz =\int_{\frac{1}{3}}^{\frac{2}{3}}\left|\partial_Z\theta_1^0\left(t,\frac{z}{\sqrt{\nu_3}}\right)\right||u_1^{err}|\frac{\langle Z\rangle ^2}{\langle Z\rangle ^2}dz\\
    &\leq \left(\int_{\frac{1}{3}}^{\frac{2}{3}}\frac{1}{\langle Z\rangle^4}(u_1^{err}(t,z))^2dz\right)^\frac{1}{2}\left(\int_{\frac{1}{3\sqrt{\nu_3}}}^{\frac{2}{3\sqrt{\nu_3}}}(\partial_Z\theta_1^0(t,Z))^2\sqrt{\nu_3}\langle Z \rangle ^4 dZ\right)^\frac{1}{2}\\
    &\leq 9\nu_3^\frac{5}{4}||u_1^{err}||_{L^2(0,1)}||\langle Z \rangle ^2 \partial_Z\theta_1^0||_{L^2(0,\infty)}
    \end{aligned}
\end{equation*}
where we have introduced the Japanese bracket $\langle Z \rangle:=\sqrt{1+|Z|^2}$. The second term in the integral can be bounded similarly.
Therefore we have that:
\begin{equation*}
\begin{aligned}
    &\frac{1}{2}\frac{d}{dt}||u_1^{err}||_{L^2(0,1)}^2+\nu_3||\partial_zu_1^{err}||_{L^2(0,1)}^2\leq 18\nu_3||u_1^{err}||_{L^2(0,1)}(||\langle Z \rangle ^2 \partial_Z\theta_1^0||_{L^2(0,\infty)}\\
    &\qquad \qquad +||\langle Z^u \rangle ^2 \partial_{Z^u}\theta_1^0||_{L^2(0,\infty)}+||u_1^0||_{H^2(0,1)}+||\theta_1^0||_{L^2(0,\infty)}+||\theta_1^{u,0}||_{L^2(0,\infty)}).
\end{aligned}
\end{equation*}
Integrating over time and applying Young and Grönwall's inequalities yields:
\begin{equation}
\begin{aligned}
\label{u1}
    &||u_1^{err}||_{L^\infty(0,T, L^2(0,1))}+\sqrt{\nu_3}||\partial_zu_1^{err}||_{L^2((0,T)\times(0,1))}\leq
    18\nu_3((||\langle Z \rangle ^2 \partial_Z\theta_1^0||_{L^2((0,T)\times(0,+\infty))}\\
     &\qquad \qquad +||\langle Z^u \rangle ^2 \partial_{Z^u}\theta_1^0||_{L^2((0,T)\times(0,+\infty))} +||u_1^0||_{L^2(0,T,H^2(0,1))}+||\theta_1^0||_{L^2((0,T)\times(0,+\infty))}\\
     & \qquad \qquad \qquad +||\theta_1^{u,0}||_{L^2((0,T)\times(0,+\infty))})
     \leq c\nu_3.
\end{aligned}
\end{equation}
%and we have the first estimates of $u_1^{err}$.
To bound the spatial derivative term, we multiply (\ref{ERR1}) by -$\partial_{zz}u_1^{err}$ and integrate over $\Omega$:
\begin{equation*}
    \begin{aligned}
        &\frac{1}{2}\frac{d}{dt}||\partial_zu_1^{err}||_{L^2(0,1)}^2+\nu_3||\partial_{zz}u_1^{err}||_{L^2(0,1)}^2\\
        &\qquad \leq 18\nu_3||\partial_{zz}u_1^{err}||_{L^2(0,1)}(||\langle Z \rangle ^2 \partial_Z\theta_1^0||_{L^2(0,\infty)}+||\langle Z^u \rangle ^2 \partial_{Z^u}\theta_1^0||_{L^2(0,\infty)}\\
    &\qquad \qquad \qquad+||u_1^0||_{H^2(0,1)}+||\theta_1^0||_{L^2(0,\infty)}+||\theta_1^{u,0}||_{L^2(0,\infty)})\leq c\nu_3||\partial_{zz}u_1^{err}||_{L^2(0,1)}\\
        &\qquad \leq \nu_3\left(\frac{||\partial_{zz}u_1^{err}||_{L^2(0,1)}^2}{2}+\frac{c^2}{2}\right),
    \end{aligned}
\end{equation*}
where we used Young's inequality and  estimate (\ref{u1}).

 Finally, integrating over time gives that 
\begin{equation}
    \begin{aligned}
    \label{u1z}
        ||\partial_zu_1^{err}||_{L^\infty(0,T, L^2(0,1))}+\sqrt{\nu_3}||\partial_{zz}u_1^{err}||_{L^2((0,T)\times(0,1))}\leq c\sqrt{\nu_3}.
    \end{aligned}
\end{equation}
Using (\ref{u1}), (\ref{u1z}) and the Anisotropic Sobolev Embedding (see Appendix \ref{s:rates}), we also obtain that:
\begin{equation}  \label{u1inf}||u_1^{err}||_{L^\infty((0,T)\times(0,1))}\leq||u_1^{err}||_{L^\infty(0,T,L^2(0,1))}^\frac{1}{2}||u_1^{err}||_{L^\infty(0,T, H^1(0,1))}^\frac{1}{2}\leq c\nu_3^\frac{3}{4}.
\end{equation}
As expected, the bounds on  $u_1^{err}$  does not depend to $\nu_1$. 

 We proceed by performing a  similar  analysis for $u_2^{err}$. Multiplying (\ref{ERR2}) by $u_2^{err}$ and integrating over $\Omega$ yields:
\begin{equation}
    \begin{aligned}
    \label{u21}
        &\frac{1}{2}\frac{d}{dt}||u_2^{err}||_{L^2(\Omega)}^2+\nu_3||\partial_zu_2^{err}||_{L^2(\Omega)}^2+\nu_1||\partial_xu_2^{err}||_{L^2(\Omega)}^2\\
        =-\int_{\Omega}u_1^{err}&\partial_x\tilde{u}_2^{app}u_2^{err}-\int_{\Omega}u_1\partial_xu_2^{err}u_2^{err}-\int_{\Omega}Cu_2^{err}-\int_{\Omega}Du_2^{err}-\int_{\Omega}Eu_2^{err}\\
        &=:J_1+J_2+J_3+J_4,
    \end{aligned}
\end{equation}
where we used that the second integral vanishes by periodicity and since $u_1^\nu$ does not depend on $x$. In particular we have:
\begin{equation*}
    \begin{aligned}
        &J_1%=\int_{\Omega}u_1^{err}\partial_x\tilde{u}_2^{app}u_2^{err}=-\int_{\Omega}[u_1^{err}(\partial_xu_2^0+\partial_x\theta_2^0+\partial_x\theta_2^{u,0})]\\
        \leq ||u_1^{err}||_{L^2(0,1)}(||\partial_xu_2^0||_{L^\infty(\Omega)}+||\partial_x\theta_2^0||_{L^\infty(\Omega_{\infty})}+||\partial_x\theta_2^{u,0}||_{L^\infty(\Omega_\infty)})||u_2^{err}||_{L^2(\Omega)}\\
        &\leq c\nu_3(||\partial_xu_2^0||_{L^\infty(\Omega)}+||\partial_x\theta_2^0||_{L^\infty(\Omega_{\infty})}+||\partial_x\theta_2^{u,0}||_{L^\infty(\Omega_\infty)})||u_2^{err}||_{L^2(\Omega)}\leq c\nu_3||u_2^{err}||_{L^2(\Omega)},
    \end{aligned}
\end{equation*}
using again  estimate  (\ref{u1}) on $u_1^{err}$,
\begin{equation*}
    \begin{aligned}
        &J_4%=-\int_{\Omega}Fu_2^{err}=\int_{\Omega}\nu_3\partial_{zz}u_2^0u_2^{err}+\int_{\Omega}\nu_3\psi''(z)\theta_2^0u_2^{err}+\int_{\Omega}\nu_3\psi''(1-z)\theta_2^{u,0}u_2^{err}\\
        \leq \nu_3 ||u_2^{err}||_{L^2(\Omega)}(||u_2^0||_{H^2(\Omega)}+||\theta_2^0||_{L^2(\Omega_\infty)}+||\theta_2^{u,0}||_{L^2(\Omega_\infty)})\leq c\nu_3||u_2^{err}||_{L^2(\Omega)},\\
        &J_3%=-\int_{\Omega}Eu_2^{err}=2\sqrt{\nu_3}\int_{\Omega}\left(\psi'(z)\partial_Z\theta_2^0\left(t,x,\frac{z}{\sqrt{\nu_3}}\right)+\psi'(1-z)\partial_{Z^u}\theta_2^{u,0}\left(t,x,\frac{1-z}{\sqrt{\nu_3}}\right)\right)u_2^{err}\\
        \leq 2\nu_3^\frac{3}{4}||u_2^{err}||_{L^2(\Omega)}(||\partial_x\theta_2^0||_{L^2(\Omega_\infty)}+||\partial_{Z^u}\theta_2^{u,0}||_{L^2(\Omega_\infty)})\leq c\nu_3^\frac{3}{4}||u_2^{err}||_{L^2(\Omega)},
    \end{aligned}
\end{equation*}
and finally, 
\begin{equation}
\begin{aligned}
\label{j2}
    &J_2=-\int_{\Omega}Cu_2^{err}=\int_{\Omega}\nu_1\partial_{xx}u_2^0u_2^{err}+\int_{\Omega}\nu_1\partial_{xx}\theta_2^0\psi(z)u_2^{err}+\int_{\Omega}\nu_1\partial_{xx}\theta_2^{u,0}\psi(1-z)u_2^{err}\\
    &\qquad -\int_{\Omega}\psi(z)(\psi(z)-1)\theta_1^0\partial_x\theta_2^0u_2^{err}-\int_{\Omega}\psi(1-z)(\psi(1-z)-1)\theta_1^{u,0}\partial_x\theta_2^{u,0}u_2^{err} \\
    &\qquad \qquad -\int_{\Omega}(\psi(z)(u_1^0-u_1^0(t,0))\partial_x\theta_2^0+\psi(z)\theta_1^0(\partial_xu_2^0-\partial_x u_2^0(t,x,0))\\
    &\qquad \qquad \qquad \int_\Omega\psi(1-z)(u_1^0-u_1^0(t,1))\partial_x\theta_2^{u,0}+\psi(1-z)\theta_1^{u,0}(\partial_xu_2^0-\partial_xu_2^0(t,x,1)))u_2^{err}.
\end{aligned}
\end{equation}
We bound the terms on the right-hand side of (\ref{j2}) as follows:
\begin{equation*}
    \begin{aligned}
        (1)\quad &-\int_{\Omega}\psi(z)(\psi(z)-1)\theta_1^0\partial_x\theta_2^0u_2^{err}\leq \int_0^L\int_\frac{1}{3}^\frac{2}{3}\left|\theta_1^0\left(t,\frac{z}{\sqrt{\nu_3}}\right)\partial_x\theta_2^0\left(t,x,\frac{z}{\sqrt{\nu_3}}\right)u_2^{err}\right| \frac{\langle Z \rangle ^4}{\langle Z \rangle ^4} dzdx\cdot \\
        & \qquad\qquad \cdot \left(\int_0^L\int_\frac{1}{3}^\frac{2}{3}\frac{|u_2^{err}|^2}{\langle Z \rangle ^4}dzdx\right)^\frac{1}{2}\left(\int_0^L\int_\frac{1}{3}^\frac{2}{3}|\theta_1^0\partial_x\theta_2^0|^2\langle Z \rangle ^4\sqrt{\nu_3}dZdx\right)^\frac{1}{2}\\
        &\leq 9\nu_3^\frac{5}{4}||u_2^{err}||_{L^2(\Omega)}||\theta_1^0\partial_x\theta_2^0\langle Z \rangle ^2||_{L^2(\Omega_\infty)}\leq 9\nu_3||u_2^{err}||_{L^2(\Omega)}||\theta_1^0||_{L^\infty(0,+\infty)}||\partial_x\theta_2^0\langle Z \rangle ^2||_{L^2(\Omega_\infty)};
    \end{aligned}
    \end{equation*}
    \begin{equation*}
    \begin{aligned}
     (2)\quad   &-\int_{\Omega}(\psi(z)(u_1^0(t,z)-u_1^0(t,0))\partial_x\theta_2^0u_2^{err}=-\int_\Omega\psi(z)z\partial_zu_1^0(t,\xi)\partial_x\theta_2^0u_2^{err}\\
        &\leq \sqrt{\nu_3}\int_{\Omega}|Z\partial_zu_1^0(t,\xi)\partial_x\theta_2^0u_2^{err}|\sqrt{\nu_3}dZdx \leq \nu_3^\frac{3}{4}||u_2^{err}||_{L^2(\Omega)}||\langle Z \rangle\partial_x\theta_2^0||_{L^2(\Omega_\infty)}||\partial_zu_1^0||_{L^\infty(0,1)}\\
        &\leq \nu_3^\frac{3}{4}||u_2^{err}||_{L^2(\Omega)}||\langle Z \rangle\partial_x\theta_2^0||_{L^2(\Omega_\infty)}||u_1^0||_{H^2(0,1)},
    \end{aligned}
\end{equation*}
thanks to the regularity of $u_1^0$;
\begin{equation*}
    \begin{aligned}
            (3) \quad    &-\int_{\Omega}(\psi(z)\theta_1^0(\partial_xu_2^0-\partial_xu_2(t,x,0))u_2^{err}=-\int_\Omega\psi(z)z\theta_1^0\partial_z\partial_xu_2^0(t,x,\xi)u_2^{err}\\
        &\leq \sqrt{\nu_3}\int_{\Omega}|Z\theta_1^0\partial_z\partial_xu_2^0(t,x,\xi)u_2^{err}|\sqrt{\nu_3}dZdx \leq \nu_3^\frac{3}{4}||u_2^{err}||_{L^2(\Omega)}||\langle Z \rangle\theta_1^0||_{L^2(0,+\infty)}||\partial_z\partial_xu_2^0||_{L^\infty(\Omega)}\\
        &\leq \nu_3^\frac{3}{4}||u_2^{err}||_{L^2(\Omega)}||\langle Z \rangle\theta_1^0||_{L^2(0,\infty)}||u_2^0||_{H^3(\Omega)},
    \end{aligned}
\end{equation*}
thanks to the regularity of $u_2^0$. Similar estimates hold for the remaining terms. 
By inserting $(1), (2), (3)$ in (\ref{j2}), it follow that 
\begin{equation*}
    \begin{aligned}
        &J_2\leq \nu_1||u_2^{err}||_{L^2(\Omega)}(||u_2^0||_{H^2(\Omega)}+||\partial_{xx}\theta_2^0||_{L^2(\Omega_\infty)}+||\partial_{xx}\theta_2^{u,0}||_{L^2(\Omega_\infty)})\\
        &+9\nu_3^\frac{3}{4}||u_2^{err}||_{L^2(\Omega)}(||\theta_1^0||_{L^\infty(0,\infty
        )}||\partial_x\theta_2^0\langle Z \rangle ^2||_{L^2(\Omega_\infty)}+||\theta_1^{u,0}||_{L^\infty(0,+\infty)}||\partial_x\theta_2^{u,0}\langle Z \rangle ^2||_{L^2(\Omega_\infty)}\\
        &+||\langle Z \rangle\partial_x\theta_2^0||_{L^2(\Omega_\infty)}||u_1^0||_{H^2(0,1)}+||\langle Z \rangle\theta_1^0||_{L^2(0,+\infty)}||u_2^0||_{H^3(\Omega)}\\
        &+||\langle Z \rangle\partial_x\theta_2^{u.0}||_{L^2(\Omega_\infty)}||u_1^0||_{H^2(0,1)}+||\langle Z \rangle\theta_1^{u,0}||_{L^2(0,+\infty)}||u_2^0||_{H^3(\Omega)})\leq c(\nu_1+\nu_3^\frac{3}{4})||u_2^{err}||_{L^2(\Omega)}.
    \end{aligned}
\end{equation*}\\
Finally, exploiting the bounds on $J_i$, $i=1,\ldots,4$ in (\ref{u21}) and applying Young and Grönwall's inequalities gives
\begin{equation*}
\label{u22}
    ||u_2^{err}||_{L^\infty(0,T; L^2(\Omega))}+\sqrt{\nu_3}||\partial_zu_2^{err}||_{L^2(0,T;L^2(\Omega))}+\sqrt{\nu_1}||\partial_xu_2^{err}||_{L^2(0,T;L^2(\Omega))}\leq c(\nu_3^\frac{3}{4}+\nu_1)
\end{equation*}
%we obtain a first estimate of $u_2^{err}$ depending also from $\nu_1$.

 Next, by multiplying (\ref{ERR2}) by $-\partial_{xx}u_2^{err}$ and integrating over $\Omega$ we have that
\begin{equation}
    \begin{aligned}
    \label{u2x1}
        &\frac{1}{2}\frac{d}{dt}||\partial_xu_2^{err}||_{L^2(\Omega)}^2+\nu_1||\partial_{xx}u_2^{err}||_{L^2(\Omega)}+\nu_3||\partial_{zx}u_2^{err}||_{L^2(\Omega)}=\\
        &\int_{\Omega}u_1^{err}\partial_x\tilde{u}_2^{app}\partial_{xx}u_2^{err}+\int_{\Omega}u_1^\nu\partial_xu_2^{err}\partial_{xx}u_2^{err}+\int_{\Omega}C\partial_{xx}u_2^{err}+\int_{\Omega}D\partial_{xx}u_2^{err}+\int_{\Omega}E\partial_{xx}u_2^{err}\\
        &=:K_1+K_2+K_3+K_4,
    \end{aligned}
\end{equation}
where again  the second integral vanishes  due to periodicity and since $u_1^\nu$ is independent  of  $x$.
Integrating by parts in $K_1, K_3, K_4$ gives:
\begin{equation*}
    \begin{aligned}  &K_1%=\int_{\Omega}u_1^{err}\partial_x\tilde{u}_2^{app}\partial_{xx}u_2^{err}=\int_{\Omega}u_1^{err}\partial_{xx}\tilde{u}_2^{app}\partial_{x}u_2^{err}\\
       \leq||u_1^{err}||_{L^2(0,1)}||\partial_xu_2^{err}||_{L^2(\Omega)}(||\partial_{xx}u_2^0||_{L^\infty(\Omega)}+||\partial_{xx}\theta_2^0||_{L^\infty(\Omega_\infty)}+||\partial_{xx}\theta_2^{u,0}||_{L^\infty(\Omega_\infty)}\\
        &\leq c\nu_3||\partial_xu_2^{err}||_{L^2(\Omega)}(||\partial_{xx}u_2^0||_{L^\infty(\Omega)}+||\partial_{xx}\theta_2^0||_{L^\infty(\Omega_\infty)}+||\partial_{xx}\theta_2^{u,0}||_{L^\infty(\Omega_\infty)})\\
        &\leq c\nu_3||\partial_xu_2^{err}||_{L^2(\Omega)},\\ 
                &K_4%=-\int_{\Omega}\nu_3(\partial_{zz}u_2^0\partial_{xx}u_2^{err}+\psi''(z)\theta_2^0\partial_{xx}u_2^{err}+\psi''(1-z)\theta_2^{u,0}\partial_{xx}u_2^{err})\\
                %&\int_{\Omega}\nu_3(\partial_{xzz}u_2^0\partial_{x}u_2^{err}+\psi''(z)\partial_x\theta_2^0\partial_{x}u_2^{err}+\psi''(1-z)\partial_x\theta_2^{u,0}\partial_{x}u_2^{err})\\
        \leq \nu_3 ||\partial_xu_2^{err}||_{L^2(\Omega)}(||u_2^0||_{H^3(\Omega)}+||\partial_x\theta_2^0||_{L^2(\Omega_\infty)}+||\partial_x\theta_2^{u,0}||_{L^2(\Omega_\infty)})\\
        &\leq c\nu_3||\partial_xu_2^{err}||_{L^2(\Omega)},\\ 
        &K_3%=-2\sqrt{\nu_3}\int_{\Omega}\left(\psi'(z)\partial_Z\theta_2^0\left(t,x,\frac{z}{\sqrt{\nu_3}}\right)+\psi'(1-z)\partial_{Z^u}\theta_2^{u,0}\left(t,x,\frac{1-z}{\sqrt{\nu_3}}\right)\right)\partial_{xx}u_2^{err}\\
        %&2\sqrt{\nu_3}\int_{\Omega}\left(\psi'(z)\partial_{xZ}\theta_2^0\left(t,x,\frac{z}{\sqrt{\nu_3}}\right)+\psi'(1-z)\partial_{xZ^u}\theta_2^{u,0}\left(t,x,\frac{1-z}{\sqrt{\nu_3}}\right)\right)\partial_{x}u_2^{err}\\
        \leq 2\nu_3^\frac{3}{4}||\partial_xu_2^{err}||_{L^2(\Omega)}(||\partial_{xZ}\theta_2^{0}||_{L^2(\Omega_\infty)}+||\partial_{xZ^u}\theta_2^{u,0}||_{L^2(\Omega_\infty)})\leq c\nu_3^\frac{3}{4}||\partial_xu_2^{err}||_{L^2(\Omega)}.
    \end{aligned}
\end{equation*}
 Integrating by parts in $K_2$ and using the same techniques as for $J_2$:
\begin{equation*}
    \begin{aligned}
           &K_2%=-\int_{\Omega}\nu_1\partial_{xx}u_2^0\partial_{xx}u_2^{err}-\int_{\Omega}\nu_1\partial_{xx}\theta_2^0\psi(z)\partial
           %_{xx}u_2^{err}-\int_{\Omega}\nu_1\partial_{xx}\theta_2^{u,0}\psi(1-z)\partial_{xx}u_2^{err}\\
    %&+\int_{\Omega}\psi(z)(\psi(z)-1)\theta_1^0\partial_x\theta_2^0\partial_{xx}u_2^{err}+\int_{\Omega}\psi(1-z)(\psi(1-z)-1)\theta_1^{u,0}\partial_x\theta_2^{u,0}\partial_{xx}u_2^{err}\\
    %&-\int_{\Omega}(\psi(z)(u_1^0-u_1^0(t,0))\partial_x\theta_2^0-\psi(z)\theta_1^0(\partial_xu_2^0-\partial_x u_2^0(t,x,0))\\
    %&-\psi(1-z)(u_1^0-u_1^0(t,1))\partial_x\theta_2^{u,0}-\psi(1-z)\theta_1^{u,0}(\partial_xu_2^0-\partial_xu_2^0(t,x,1)))\partial_{xx}u_2^{err}\\
        \leq \nu_1||\partial_xu_2^{err}||_{L^2(\Omega)}(||u_2^0||_{H^3(\Omega)}+||\partial_{xxx}\theta_2^0||_{L^2(\Omega_\infty)}+||\partial_{xxx}\theta_2^{u,0}||_{L^2(\Omega_\infty)})\\
        &+9\nu_3^\frac{3}{4}||\partial_xu_2^{err}||_{L^2(\Omega)}(||\theta_1^0||_{L^\infty(0,+\infty)}||\partial_{xx}\theta_2^0\langle Z \rangle ^2||_{L^2(\Omega_\infty)}+||\theta_1^{u,0}||_{L^\infty(0,+\infty)}||\partial_{xx}\theta_2^{u,0}\langle Z \rangle ^2||_{L^2(\Omega_\infty)}\\
        &+||\langle Z \rangle\partial_x\theta_2^0||_{L^2(\Omega_\infty)}||u_1^0||_{H^3(0,1)}+||\langle Z \rangle\theta_1^0||_{L^2(0,\infty)}||u_2^0||_{H^4(\Omega)}\\
        &+||\langle Z \rangle\partial_x\theta_2^{u.0}||_{L^2(\Omega_\infty)}||u_1^0||_{H^3(0,1)}+||\langle Z \rangle\theta_1^{u,0}||_{L^2(0,\infty)}||u_2^0||_{H^4(\Omega)})\\
        &\leq c(\nu_1+\nu_3^\frac{3}{4})||\partial_xu_2^{err}||_{L^2(\Omega)}.
    \end{aligned}
\end{equation*}
By inserting these bounds in (\ref{u2x1}) and using Young and Grönwall's inequalities we have that:
\begin{equation*}
    \label{u2x2}||\partial_xu_2^{err}||_{L^\infty(0,T;L^2(\Omega))}+ \sqrt{\nu_1}||\partial_{xx}u_2^{err}||_{L^2((0,T)\times\Omega)}+\sqrt{\nu_3}||\partial_{xz}u_2^{err}||_{L^2((0,T)\times\Omega)}
    \leq c(\nu_3^\frac{3}{4}+\nu_1).
\end{equation*}
Similarly, multiplying (\ref{ERR2}) by $-\partial_{zz}u_2^{err}$ yields:
\begin{equation}
    \begin{aligned}
    \label{u2z1}
                &\frac{1}{2}\frac{d}{dt}||\partial_zu_2^{err}||_{L^2(\Omega)}^2+\nu_1||\partial_{zx}u_2^{err}||_{L^2(\Omega)}^2+\nu_3||\partial_{zz}u_2^{err}||_{L^2(\Omega)}^2=\\
        &\int_{\Omega}u_1^{err}\partial_x\tilde{u}_2^{app}\partial_{zz}u_2^{err}+\int_{\Omega}u_1^\nu\partial_xu_2^{err}\partial_{zz}u_2^{err}+\int_{\Omega}C\partial_{zz}u_2^{err}+\int_{\Omega}D\partial_{zz}u_2^{err}+\int_{\Omega}E\partial_{zz}u_2^{err}\\
        &=:L_1+L_2+L_3+L_4+L_5.
    \end{aligned}
\end{equation}
We estimate the terms on the right as follows:
\begin{flalign*}
    &L_1+L_2 \leq ||u_1^{err}||_{L^2(0,1)} ||\partial_{zz}u_2^{err}||_{L^2(\Omega)}
    ( ||\partial_{x}u_2^0||_{L^\infty(\Omega)}+ ||\partial_{x}\theta_2^0||_{L^\infty(\Omega_\infty)}+ ||\partial_{x}\theta_2^{u,0}||_{L^\infty(\Omega_\infty)})\notag \\
    &\quad + ||\partial_{zz}u_2^{err}||_{L^2(\Omega)} ||u_1^\nu||_{L^\infty(\Omega)}||\partial_xu_2^{err}||_{L^2(\Omega)}\leq c \nu_3 ||\partial_{zz}u_2^{err}||_{L^2(\Omega)} 
    ( ||\partial_{x}u_2^0||_{L^\infty(\Omega)} \notag \\
    &\quad + ||\partial_{x}\theta_2^0||_{L^\infty(\Omega_\infty)} +
    ||\partial_{x}\theta_2^{u,0}||_{L^\infty(\Omega_\infty)}) 
    + ||\partial_{zz}u_2^{err}||_{L^2(\Omega)} ||u_1^\nu||_{L^\infty(\Omega)}(\nu_3^{3/4} + \nu_1)\\
    &\quad\leq c(\nu_3^\frac{3}{4}+\nu_1)||\partial_{zz}u_2^{err}||_{L^2(\Omega)}&
\end{flalign*}
using (\ref{u1}) and (\ref{u1inf}).
\begin{flalign*}
    &L_5 \leq \nu_3 ||\partial_{zz}u_2^{err}||_{L^2(\Omega)}
    ( ||u_2^0||_{H^2(\Omega)} + ||\theta_2^0||_{L^2(\Omega_\infty)} + ||\theta_2^{u,0}||_{L^2(\Omega_\infty)})\leq c\nu_3||\partial_{zz}u_2^{err}||_{L^2(\Omega)}. &\\
    &L_4 \leq 2\nu_3^{3/4} ||\partial_{zz}u_2^{err}||_{L^2(\Omega)}
    ( ||\partial_{Z}\theta_2^{0}||_{L^2(\Omega_\infty)} + ||\partial_{Z^u}\theta_2^{u,0}||_{L^2(\Omega_\infty)})\leq c\nu_3^\frac{3}{4}||\partial_{zz}u_2^{err}||_{L^2(\Omega)}. &
        \end{flalign*}
    \begin{flalign*}
    &L_3 \leq \nu_1 ||\partial_{zz}u_2^{err}||_{L^2(\Omega)} 
    ( ||u_2^0||_{H^2(\Omega)} + ||\partial_{xx}\theta_2^0||_{L^2(\Omega_\infty)} + ||\partial_{xx}\theta_2^{u,0}||_{L^2(\Omega_\infty)}) \notag \\
    &\quad + 9\nu_3^{3/4} ||\partial_{zz}u_2^{err}||_{L^2(\Omega)} 
    ( ||\theta_1^0||_{L^\infty(0,\infty)} ||\partial_{x}\theta_2^0 \langle Z \rangle^2||_{L^2(\Omega_\infty)}+ ||\theta_1^{u,0}||_{L^\infty(0,\infty)} ||\partial_{x}\theta_2^{u,0} \langle Z \rangle^2||_{L^2(\Omega_\infty)} \notag \\
    &\quad + ||\langle Z \rangle \partial_x\theta_2^0||_{L^2(\Omega_\infty)} ||u_1^0||_{H^2(0,1)} 
    + ||\langle Z \rangle \theta_1^0||_{L^2(0,\infty)} ||u_2^0||_{H^3(\Omega)} \notag \\
    &\quad + ||\langle Z \rangle \partial_x\theta_2^{u,0}||_{L^2(\Omega_\infty)} ||u_1^0||_{H^2(0,1)} 
    + ||\langle Z \rangle \theta_1^{u,0}||_{L^2(0,\infty)} ||u_2^0||_{H^3(\Omega)})\\
    &\quad\leq c(\nu_3^\frac{3}{4}+\nu_1)||\partial_{zz}u_2^{err}||_{L^2(\Omega)}&
\end{flalign*}
using the same techniques as in $J_2$.
We now insert these estimates in (\ref{u2z1}):
\begin{equation} 
    \begin{aligned}
    \label{young}
        &\frac{1}{2}\frac{d}{dt}||\partial_zu_2^{err}||_{L^2(\Omega)}^2+\nu_1||\partial_{zx}u_2^{err}||_{L^2(\Omega)}^2+\nu_3||\partial_{zz}u_2^{err}||_{L^2(\Omega)}^2\\
        &\qquad \qquad\leq c\nu_3^\frac{3}{4}||\partial_{zz}u_2^{err}||_{L^2(\Omega)}+c\nu_1||\partial_{zz}u_2^{err}||_{L^2(\Omega)}.
    \end{aligned}
\end{equation}
There are two distinct cases: $\nu_1\leq \nu_3$ and $\nu_3 < \nu_1$. The second case includes the situation where the normal viscosity vanishes faster than the tangential one, which is the situation studied already in the literature for general flows \cite{masmoudi1998euler,nonflat}. The most interesting and novel situation is when the tangential viscosity goes to zero faster than the normal viscosity in the first case.

Hence, we let $\nu_1\leq\nu_3$. Using Young's inequality in \eqref{young}:
\begin{equation*}
\begin{aligned}
    &\frac{1}{2}\frac{d}{dt}||\partial_zu_2^{err}||_{L^2(\Omega)}^2+\nu_1||\partial_{zx}u_2^{err}||_{L^2(\Omega)}^2+\nu_3||\partial_{zz}u_2^{err}||_{L^2(\Omega)}^2\leq\\
    &\qquad \qquad \frac{\nu_3}{2}||\partial_{zz}u_2^{err}||^2_{L^2(\Omega)}+\frac{\nu_1}{2}||\partial_{zz}u_2^{err}||^2_{L^2(\Omega)}+\frac{c^2\nu_3^\frac{1}{2}}{2}+\frac{c^2\nu_1}{2}.
    \end{aligned}
\end{equation*}
Consequently:
\begin{equation*}
            \frac{1}{2}\frac{d}{dt}||\partial_zu_2^{err}||_{L^2(\Omega)}^2+\nu_1||\partial_{zx}u_2^{err}||_{L^2(\Omega)}^2+\frac{1}{2}(\nu_3-\nu_1)||\partial_{zz}u_2^{err}||_{L^2(\Omega)}^2\leq c(\nu_3^\frac{1}{2}+\nu_1).
\end{equation*}
Then we have that:
\begin{equation}
    \begin{aligned}
    \label{u2z2}     &\qquad ||\partial_zu_2^{err}||_{L^\infty(0,T;L^2(\Omega))}+\sqrt{\nu_1}||\partial_{zx}u_2^{err}||_{L^2((0,T)\times\Omega)}+\sqrt{\nu_3-\nu_1}||\partial_{zz}u_2^{err}||_{L^2((0,T)\times\Omega)}\\
        &\qquad \qquad \leq c(\nu_3^\frac{1}{4}+\nu_1^\frac{1}{2}).
    \end{aligned}
\end{equation}

To tackle the case $\nu_3<\nu_1$, we apply  Young's inequality in (\ref{young}) differently:
\begin{equation*}
\begin{aligned}
        &\frac{1}{2}\frac{d}{dt}||\partial_zu_2^{err}||_{L^2(\Omega)}^2+\nu_1||\partial_{zx}u_2^{err}||_{L^2(\Omega)}^2+\nu_3||\partial_{zz}u_2^{err}||_{L^2(\Omega)}^2\\
        & \qquad \qquad \leq \frac{\nu_3}{2}||\partial_{zz}u_2^{err}||^2_{L^2(\Omega)}+\frac{\nu_1^{2\alpha}}{2}||\partial_{zz}u_2^{err}||^2_{L^2(\Omega)}+\frac{c^2\nu_3^\frac{1}{2}}{2}+\frac{c^2\nu_1^{2(1-\alpha)}}{2}
\end{aligned}
\end{equation*}
with $0<\alpha<1$, and thus:
\begin{equation*}
            \frac{1}{2}\frac{d}{dt}||\partial_zu_2^{err}||_{L^2(\Omega)}^2+\nu_1||\partial_{zx}u_2^{err}||_{L^2(\Omega)}^2+\frac{1}{2}(\nu_3-\nu_1^{2\alpha})||\partial_{zz}u_2^{err}||_{L^2(\Omega)}^2\leq c(\nu_3^\frac{1}{2}+\nu_1^{2(1-\alpha)}).
\end{equation*}
It follows that, to achieve convergence,  $\nu_3<\nu_1 \leq \nu_3^\frac{1}{2\alpha}$, where $\frac{1}{2}<\alpha<1$, with the optimal rate for $\alpha \approx 1$.
We then conclude  that:
\begin{equation}
    \begin{aligned}
    \label{u2z22} &||\partial_zu_2^{err}||_{L^\infty(0,T;L^2(\Omega))}+\sqrt{\nu_1}||\partial_{zx}u_2^{err}||_{L^2((0,T)\times\Omega)}+\sqrt{\nu_3-\nu_1^{2\alpha}}||\partial_{zz}u_2^{err}||_{L^2((0,T)\times\Omega)}\\
        &\qquad \qquad \leq c(\nu_3^\frac{1}{4}+\nu_1^{{1-\alpha}}).
    \end{aligned}
\end{equation}
Similarly, we  can show$||\partial_{xx}u_2^{err}||_{L^2(\Omega)}$ satisfies the same estimate as in (\ref{u2x2}) and $||\partial_{xz}u_2^{err}||_{L^2(\Omega)}$  the same bounds as in  (\ref{u2z2}) and (\ref{u2z22}). By the Anisotropic Sobolev Embedding (see again  Appendix \ref{s:rates}), we obtain the desired estimate in $L^\infty((0,T)\times\Omega)$ norm. 
More precisely, if $\nu_1\leq\nu_3$,
\begin{equation}
\begin{aligned}
\label{u2inf1}
    &||u_2^{err}||_{L^\infty((0,T)\times\Omega)}\leq c(||u_2^{err}||_{L^\infty(0,T;L^2(\Omega))}^\frac{1}{2}||\partial_zu_2^{err}||_{L^\infty(0,T,L^2(\Omega))}^\frac{1}{2}+||\partial_zu_2^{err}||_{L^\infty(0,T; L^2(\Omega))}^\frac{1}{2}\\
    &\qquad \qquad ||\partial_xu_2^{err}||_{L^\infty(0,T; L^2(\Omega))}^\frac{1}{2}+||u_2^{err}||_{L^\infty(0,T;L^2(\Omega))}^\frac{1}{2}||\partial_{xz}u_2^{err}||_{L^\infty(0,T; L^2(\Omega))})\\
    &\qquad \qquad \qquad \leq c(\nu_3^\frac{1}{2}+\nu_1^\frac{3}{4}+\nu_1^\frac{1}{4}\nu_3^\frac{3}{8}+\nu_1^\frac{1}{2}\nu_3^\frac{1}{8}).
    \end{aligned}
\end{equation}
If $\nu_3 < \nu_1$,
\begin{equation}
\begin{aligned}
\label{u2inf2}
    &||u_2^{err}||_{L^\infty((0,T)\times\Omega)}\leq c(||u_2^{err}||_{L^\infty(0,T;L^2(\Omega))}^\frac{1}{2}||\partial_zu_2^{err}||_{L^\infty(0,T,L^2(\Omega))}^\frac{1}{2}+||\partial_zu_2^{err}||_{L^\infty(0,T; L^2(\Omega))}^\frac{1}{2}\\
    &\qquad \qquad ||\partial_xu_2^{err}||_{L^\infty(0,T; L^2(\Omega))}^\frac{1}{2}+||u_2^{err}||_{L^\infty(0,T;L^2(\Omega))}^\frac{1}{2}||\partial_{xz}u_2^{err}||_{L^\infty(0,T; L^2(\Omega))})\\
    &\qquad \qquad \qquad \leq c(\nu_3^\frac{1}{2}+\nu_1^\frac{2-\alpha}{2}+\nu_1^\frac{1-\alpha}{2}\nu_3^\frac{3}{8}+\nu_1^\frac{1}{2}\nu_3^\frac{1}{8}).
    \end{aligned}
\end{equation}
We have finally:
\begin{equation*}
\begin{aligned}
    &||\mathbf{u}^\nu-\mathbf{u}^0||_{L^\infty(0,T;L^2(\Omega))}\leq ||\mathbf{u}^{err}||_{L^\infty(0,T; L^2(\Omega))}+||\bm \theta^0||_{L^\infty(0,T;L^2(\Omega_\infty))}+||\bm \theta^{u,0}||_{L^\infty(0,T;L^2(\Omega_\infty))}\\
    & \qquad \qquad \leq c(\nu_3^\frac{1}{4}+\nu_1).
    \end{aligned}
\end{equation*}
Theorem \ref{teo1} is now proved.
In particular, the vanishing viscosity limit holds under the stated conditions on $\nu_1, \nu_3$.
\end{proof}

\subsection{Vorticity estimates}  \label{ss:PlaneVorticity}

In this subsection we study the behavior of the vorticity $\bm \omega:=\nabla \times \mathbf{u}$ in the boundary layer. Vorticity controls the behavior of the fluid in the layer and possible layer separation. It also generates instabilities in the flow that can lead to onset of turbulence, such as the turbulence observe in wake of a moving obstacle.

We show below that, owing to the flow symmetry, vorticity produced by the approximate viscous flow remains close to the vorticity
of the true viscous flow as viscosity vanishes. To this end, we define $ \bm \omega^{err}$ as the curl of $\mathbf{u}^{err}$ and note that 
\begin{equation}
    \bm \omega^{err}=\nabla \times \mathbf{u}^{err}:=(-\partial_zu_2^{err},\partial_zu_1^{err},\partial_xu_2^{err})=(\omega_1,\omega_2,\omega_3)
\end{equation}
using the plane parallel symmetry.
We then obtain the equation for $\omega^{err}_1$ by differentiating  (\ref{ERR2}) with respect to $z$, the equation for $\omega^{err}_2$ by differentiating (\ref{ERR1}) with respect to $z$ and the equation for $\omega_3^{err}$ by differentiating (\ref{ERR2}) with respect to $x$:
\begin{equation}
\label{vorticity}
\begin{cases}
    \partial_t\omega_2-\nu_3\partial_{zz}\omega_2=\partial_z(A+B)\\
    \partial_t\omega_3+u_1^{err}\partial_{xx}\tilde{u}_2^{app}+u_1\partial_x\omega_3-\nu_1\partial_{xx}\omega_3-\nu_3\partial_{zz}\omega_3=-\partial_x(C+D+E)\\
    \partial_t\omega_1+u_1\partial_x\omega_1-\nu_1\partial_{xx}\omega_1-\nu_3\partial_{zz}\omega_1\\    =\omega_2\partial_x\tilde{u_2}^{app}+u_1^{err}+\partial_{xz}\tilde{u_2}^{app}+\partial_zu_1\omega_3+\partial_z(C+D+E)
    \end{cases}
\end{equation}
complemented by the initial condition
\begin{equation*}
    \bm\omega^{err}|_{t=0}=0.
\end{equation*}

 One of the main difficulties in working with the vorticity in bounded domains is deriving correct and useful boundary conditions, given the non-local nature of the relation between vorticity and velocity or pressure or potential. Under the symmetry considered in this work, it was already observed in \cite{GIE20191237} and references therein that the vorticity boundary conditions are of mixed Dirichlet-Neumann type and can be derived by restricting the equations of motion to the boundary for strong solutions.
 Similarly here, restricting (\ref{ERR1}) and (\ref{ERR2}) to $z=0,z=1$ gives the following boundary conditions:
\begin{align}
    \omega_3|_{z=0}&=\omega_3|_{z=1}=0, \nonumber \\
    \partial_z\omega_2&=-\partial_{zz}u_1^0 \quad z=0,z=1,  \nonumber \\
    \partial_z\omega_1&=\frac{\nu_1}{\nu_3}\partial_{xx} u_2^0+\partial_{zz}u_2^0 \quad z=0,z=1,  \nonumber
\end{align}

 If $\nu_1\leq \nu_3$, using (\ref{u1z}), (\ref{u2x2}) and (\ref{u2z2}), we then have the following estimates:
\begin{align}
    \label{vor2}
        &||\omega_2||_{L^\infty(0,T, L^2(0,1))}+\sqrt{\nu_3}||\partial_{z}\omega_2||_{L^2((0,T)\times(0,1))}\leq c\sqrt{\nu_3}, \\
\label{vor3}
    &||w_3||_{L^\infty(0,T;L^2(\Omega))}+ \sqrt{\nu_1}||\partial_{x}\omega_3||_{L^2((0,T)\times\Omega)}+\sqrt{\nu_3}||\partial_{z}\omega_3||_{L^2((0,T)\times\Omega)}
    \leq c(\nu_3^\frac{3}{4}+\nu_1), \\
    \label{vor1}      &||\omega_1||_{L^\infty(0,T;L^2(\Omega))}+\sqrt{\nu_1}||\partial_{x}\omega_1||_{L^2((0,T)\times(\Omega))}+\sqrt{\nu_3-\nu_1}||\partial_{z}\omega_2||_{L^2((0,T)\times\Omega))}\\
        &\qquad \qquad\leq c(\nu_3^\frac{1}{4}+\nu_1^\frac{1}{2}).
 \end{align}
Consequently:
\begin{equation}
\label{vor}
    ||\bm \omega^{err}||_{L^\infty(0,T;L^2(\Omega))}\leq c(\nu_3^\frac{1}{4}+\nu_1^\frac{1}{2}).
\end{equation}

 If $\nu_3 < \nu_1 < \nu_3^\frac{1}{2\alpha}$  with $\frac{1}{2}<\alpha<1$, the estimate for $\omega_2$ and $\omega_3$ are the same, while the estimate for $\omega_1$ becomes:
\begin{equation*}
    \begin{aligned}
    %\label{vor12}  
&||\omega_1||_{L^\infty(0,T;L^2(\Omega))}+\sqrt{\nu_1}||\partial_{x}\omega_1||_{L^2((0,T)\times\Omega)}+\sqrt{\nu_3-\nu_1^{2\alpha}}||\partial_{z}\omega_2||_{L^2((0,T)\times\Omega))}\\
        &\qquad \qquad \leq c(\nu_3^\frac{1}{4}+\nu_1^{1-\alpha}).
    \end{aligned}
\end{equation*}
Therefore,
\begin{equation}
\label{vorr}
    ||\bm \omega^{err}||_{L^\infty(0,T;L^2(\Omega))}\leq c(\nu_3^\frac{1}{4}+\nu_1^{1-\alpha}).
\end{equation}

Next, we investigate the convergence of the vorticity in the zero-viscosity limit. 
In the isotropic case, it was shown in \cite{kelliher2008vanishing} (see also \cite{kelliher2017observations}) that the validity of the zero-viscosity limit  is equivalent to weak convergence  of $\bm \omega^\nu$ in $L^\infty(0,T;(H^1(\Omega))')$ to $\bm \omega^0$ plus a distribution on the boundary. 
Indeed, due the mismatch in boundary conditions, one cannot expect strong convergence of the vorticity. Furthermore, because of the Sobolev Embedding Theorem  and the Trace Theorem, the only possible bound on vorticity uniformly in viscosity is an $L^1$ bound. This bound in turn, if valid, implies weak convergences in the space of Radon measures on $\Omega$ \cite{mazzucato2008vanishing,GIE20191237}. We establish here the analogous result, which shows that vorticity can only accumulate as a measure on the boundary in the limit.

We recall that, with abuse of notation, we identify plane-parallel vectors  on the periodic channels $\cC$ with vectors on the reduced two-dimensional domain $\Omega$.

\begin{theorem} \label{t:vorticityPlane}
    Assume $\mathbf{u}_0 \in H^m(\cC)$, $m>5$, $\mathbf{f} \in C^\infty([0,T]\times\cC)$ with the compatibility conditions (\ref{CC1}) and assume the viscosity coefficients  $\nu_1$ and $\nu_3$ satisfy $\nu_1\leq \nu_3$ or $\nu_3\leq \nu_1\leq \nu_3^\frac{1}{2\alpha}$ with $\frac{1}{2}<\alpha<1$. Then,
\begin{equation}
\label{weakvort}
   \bm \omega ^\nu \rightharpoonup \bm \omega^0 + (\mathbf{ u}^0\times \mathbf{n})\mu \quad \textit{as} \quad \nu_1,\nu_3\to 0
\end{equation}
in $L^\infty(0,T;\mathcal{M}(\bar{\cC}))$, where $\mathcal{M}(\bar \cC)$ is the space of Radon measures on $\bar \cC$, $\mathbf{n}$ is the unit outer normal at the boundary, and $\mu$ is the unitary measure on the boundary.
\end{theorem}

\begin{proof} 
Weak convergence of $\bm \omega^\nu$  to the same limit in the sense of distributions follows from the validity of the zero-viscosity limit. By uniqueness of the limit, it is enough to show that $\bm \omega^\nu$ is uniformly bounded in $L^1(\cC)$ uniformly in time and viscosity.

Using higher-order estimates on the correctors (see Appendix \ref{s:rates}), we obtain that:
\begin{equation*}
   ||\nabla \times \bm \theta^0||_{L^\infty(0,T;L^1(\Omega))}\leq c(T,||u_1^0||_{L^\infty(0,T;H^1(\Omega))},||u_2^0||_{L^\infty(0,T;H^2(\Omega))}).
\end{equation*}
and the same is valid for $\bm \theta^{u,0}$.

We can then conclude that: 
\begin{equation*}
\label{vort}
    ||\bm \omega^\nu -\bm \omega^0||_{L^\infty(0,T;L^1(\cC))}\leq c(T,||u_1^0||_{L^\infty(0,T;H^1(\Omega))},||u_2^0||_{L^\infty(0,T;H^2(\Omega))}),
\end{equation*}
which gives the desired result.
%corollary 4.1
\end{proof}

\section{Pipe Parallel Flow} \label{s:pipe}

We now consider a different class of symmetric flows, the so-called pipe parallel flows, which can be thought of as plane-parallel flows in the curved geometry of a straight circular pipe. It is well known that boundary curvature can affect the behavior in the boundary layer and this phenomenon has been rigorously studied at least in the linearized case \cite{Gie2014,2018JMFM...20.1405G}.

We look for solutions of the fluid equations of the form:
\begin{equation} \label{eq:PipeSymmetry}
    \mathbf{u}(t,r,\phi,x)=u_\phi(t,r)\mathbf{e}_\phi+u_x(t,r,\phi)\mathbf{e}_x,
\end{equation}
using cylindrical coordinates $(x,r,\phi)$ in an infinitely long horizontal pipe, with periodicity in the $x$ direction. We consider as spatial domain $\cC=[0,L]\times \Omega$ where $L$ is the horizontal period and 
$\Omega=\{(r,\phi)| \delta\leq r \leq 1, \phi\in[0,2\pi]\}$ is the annual vertical cross-section of the pipe (see Figure \ref{f:pipe}), where $0<\delta<1$ is a fixed small parameter, independent of viscosity. All functions will be assumed periodic in $\phi$ and $x$, hence boundary conditions will only be imposed at $r=1,\delta$. Although the domain depends on $\delta$, this dependence is irrelevant here.
%Since we are concerned with strong solutions, the velocity must vanish at the axis and the pressure be constant. The existence of global strong solutions follows as for the isotropic 2D Navier-Stokes and Euler equations  (see e.g. \cite{Majda_Bertozzi_2001}).

We again consider anisotropic viscosity, that is, the viscosity coefficients in the radial, angular, and axial directions are different. Hence, the diffusion operator in (\ref{NSeq}) becomes $\nu_1\partial_{rr}+\frac{\nu_1}{r}\partial_r+\frac{\nu_2}{r^2}\partial_{\phi\phi}+\nu_3\partial_{xx}$, where $\nu_1$, $\nu_2$ and $\nu_3$ are the viscosity along $\mathbf{e}_r$, $\mathbf{e}_\phi$ and $\mathbf{e}_x$, respectively. We choose this particular relabeling, as the system is actually independent of $\nu_3$.

\begin{remark} \label{r:PipeDomain}
We choose to work with a domain that misses the axis, as our focus is the behavior in the boundary layer, since otherwise the anisotropic diffusion operator degenerates at the axis.  There are several ways to address this point. One is to use suitable weights and to restore the ellipticity of the operator in weighted spaces. We choose here instead to assume that the three coefficients smoothly become equal \ $\nu_1=\nu_2=\nu_3$ near the axis, or equivalently that the diffusion operator in Cartesian coordinates is $\nu_1 \Delta$ near the axis of the pipe. This assumption is physically reasonable in absence of rotation effects or other strongly anisotropic forces, but it leads to a variable-coefficient diffusion operator. Our results can be extended to this last setting, under suitable restriction on the coefficients $\nu_1$ and $\nu_2$ to control additional terms arising from integrating by parts in the energy estimates.
\end{remark}

%Unless the cross section is an annulus, this operator degenerates at the axis. There are several ways to address this point. One is to use suitable weights and to restore the ellipticity of the operator in weighted spaces. We choose here instead to assume that the three coefficients smoothly become equal 
%\ $\nu_1=\nu_2=\nu_3$ near the axis, or equivalently that the diffusion operator in Cartesian coordinates is $\nu_1 \Delta$ near the axis of the pipe. This assumption is physically reasonable in absence of rotation effects or other strongly anisotropic forces, but it leads to a variable-coefficient diffusion operator. This situation does not impact the boundary layer analysis, which is performed locally near the boundary. However, more care has to be taken when using global estimates, as in Section \ref{s:semigroup} and when estimating higher norms. In Section \ref{s:semigroup} dealing with planar circularly symmetric flows, for simplicity we choose to work in an annular domain.

\begin{figure}[ht]
    \centering
\begin{tikzpicture}[scale=0.8]

% Parameters
\def\L{8}      % length
\def\R{1.8}    % radius
\def\shift{0.8}
\def\Ri{0.3}   % raggio interno

\draw[thick, line width=2 pt] (0,0) ellipse [x radius=\shift, y radius=\R];

% Back half (hidden)
\draw[dashed, thick]
(\L,\R)
arc[start angle=90,end angle=270,
    x radius=\shift,
    y radius=\R];

% Front half (visible)
\draw[thick, line width=2 pt]
(\L,-\R)
arc[start angle=-90,end angle=90,
    x radius=\shift,
    y radius=\R];

% Cylinder walls
\draw[thick, line width=2 pt] (0,\R) -- (\L,\R);
\draw[thick, line width=2 pt] (0,-\R) -- (\L,-\R);

% Axis
\draw[thick,->] (0,0) -- (\L+1.5,0) node[right] {$x$};

% Flow arrows
\foreach \x/\y in {
    0.75/-1.2,
    0.9/-0.8,
    0.9/0.7,
    0.75/1.1
}{
    \draw[->] (\x,\y) -- (\x+3,\y);
}

% Q label
\node at (4.3,1.3) {$\mathcal{C}$};

% Omega label
\node at (0,-1.25) {$\Omega$};

% Velocity field
\node at (3.7,-1.5)
{
$\mathbf{u}=u(r,t)\,\mathbf{e}_{\phi}
+u(r,\phi,t)\,\mathbf{e}_{x}$
};

% Inner cylinder (bucato)
\draw[dashed, thick]
(\L,\Ri)
arc[start angle=90,end angle=270,
    x radius=\shift*\Ri/\R,
    y radius=\Ri];

\draw[dashed, thick]
(\L,-\Ri)
arc[start angle=-90,end angle=90,
    x radius=\shift*\Ri/\R,
    y radius=\Ri];

\draw[thick, line width=2pt] (0,0) ellipse [x radius=\shift*\Ri/\R, y radius=\Ri];

\draw[dashed, thick] (0,\Ri) -- (\L,\Ri);
\draw[dashed, thick] (0,-\Ri) -- (\L,-\Ri);

\end{tikzpicture}
    \caption{Pipe Parallel Symmetry} \label{f:pipe}
\end{figure}
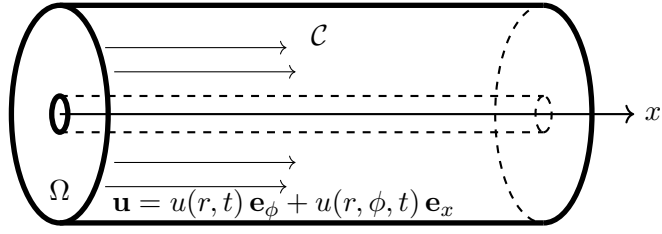

As in the channel case, pipe-parallel flows are automatically divergence free and satisfy the no-penetration condition at the boundary,  and they give exact solutions of the Navier-Stokes and Euler equations, provided the initial condition, the forcing term have also pipe parallel symmetry, while the pressure is assumed radial, that is:
\begin{equation*}
    \begin{aligned}
    &\mathbf{u}|_{t=0}(x,r,\phi)=\mathbf{u}_0(r,\phi)=a(r)\mathbf{e}_\phi+b(r,\phi)\mathbf{e}_x,\\
    &\mathbf{f}(t,x,r,\phi)=f_1(t,r)\mathbf{e}_\phi+f_2(t,r,\phi)\mathbf{e}_x, \\
    & p(t,x,r,\phi)=\tilde{p}(t,r)
\end{aligned}
\end{equation*}
at least for sufficiently regular solutions. Furthermore, we observe that, since $\mathbf{e}_x$ is a constant vector, there is no explicit dependence on $x$, and we can pose the  problem in the domain $\Omega\times [0,T]$, instead of $\cC\times [0,T].$ To account for the polar coordinate structure, we implicitly use the measure $rdr\, d\phi$, which in our setting is equivalent to the Legensgue measure, on $\Omega$ and denote $L^p(\Omega, rdr\,d\phi)$ as $L^p(\Omega)$. We use a similar notation for Sobolev spaces.

Assuming the pipe-parallel symmetry in (\ref{NSeq}), the Navier-Stokes equations reduce to the following system in a cylindrical frame:
\begin{equation}
\begin{cases}
\label{NSEPIPE}
    \partial_rp^\nu-\frac{(u^\nu_\phi)^2}{r}=0, \quad \quad \quad \quad \quad \quad \quad \quad \quad \quad \quad \quad \quad \quad \quad \quad \quad &\text{in } \Omega\times[0,T],\\
    \partial_tu_\phi^\nu-\nu_1\frac{1}{r}\partial_r(r\partial_ru_\phi^\nu)+\nu_2\frac{1}{r^2}u_\phi^\nu=f_1, &\text{in } \Omega\times[0,T],\\
    \partial_tu_x^\nu+\frac{1}{r}u_\phi^\nu\partial_\phi u_x^\nu-\nu_1\frac{1}{r}\partial_r(r\partial_ru_x^\nu)-\nu_2\frac{1}{r^2}\partial_{\phi\phi}u_x^\nu=f_2, & \text{in } \Omega\times[0,T],\\
    u^\nu_\phi(0,r)=a(r), u^\nu_x(0,r,\phi)=b(r,\phi),  & \text{on } \Omega\,\\
    \; u^\nu_\phi(t,r)=u^\nu_x(t,r,\phi)= 0, & \text{for } r=1,\delta, \; t\in  [0,T],
    \end{cases}
\end{equation}
where $\nu$ denotes collectively both $\nu_1$ and $\nu_2$.
We note that the equations for pressure and velocity are decoupled, and the pressure can be recovered from the velocity. The equations of motion reduce to a $2\times 2$ system, even though the flow is not planar. We also observe that the system is independent of  the axial viscosity $\nu_3$. 

 Formally taking the limit $\nu_1,\nu_2,\nu_3\to 0$, we obtain the following $2\times 2$ system for the inviscid flow
\begin{equation}
\begin{cases}
\label{EEPIPE}
    \frac{-(u_\phi^0)^2}{r}+\partial_rp^0=0, \quad \quad \quad \quad \quad \quad & \text{in } \cC\times[0,T],\\
    \partial_tu_\phi^0=f_1,   & \text{in } \cC\times[0,T],\\
    \partial_tu_x^0+\frac{u_\phi^0}{r}\partial_\phi u_x^0=f_2,  & \text{in } \cC\times[0,T]\,\
 %   \mathbf{u}^0(0,r,\phi)=a(r)e_\phi+b(r,\phi)e_x &\cC.
    \end{cases}
\end{equation}
with the same initial conditions as for the viscous flow.
%We observe that the non-penetration ($\mathbf{u}^0\cdot \mathbf{n}=0$) and divergence-free conditions are automatically satisfied with this symmetry.\\

%From now on, due to the independence of the system form $x$, we consider as domain $\Omega$ and we set:
We consider the velocity as a vector-valued function on the domain $\Omega\times [0,T]$:
\begin{align}  
&\mathbf{u}^\nu(t,r,\phi):=u_\phi^\nu(t,r)r_\phi+u_x^\nu(t,r,\phi)
\mathbf{e}_x, \nonumber \\
&\mathbf{u}^0(t,r,\phi):=u_\phi^0(t,r)\mathbf{e}_\phi+ u_x^0(t,r,\phi)\mathbf{e}_x, \nonumber \\
&\mathbf{u}_0(t,r,\phi):=a(r)\mathbf{e}_\phi+b(r,\phi)\mathbf{e}_x,
\end{align}
where $u^\nu_\phi$, $u^\nu_x$ solve the system (\ref{NSEPIPE}), $u^0_\phi$, $u^0_x$ solve (\ref{EEPIPE}), and $\mathbf{u}_0$ is the initial data. 

%All these functions are periodic in the $\phi$ variable (with also derivatives periodic in that direction) and boundary conditions will be given only at $r=1$. We observe that all the results will be shown in $\Omega$, but by the independence of the system from the $x$ coordinate we can extend all in the physical domain $\cC$.

To estimate higher norms, it will be convenient to study the fluid equations in Cartesian coordinates $(x,y,z)$ as well. To avoid introducing a heavy notation,  {\em we denote the periodic cylinder in a Cartesian frame again with $\cC$ and its cross-section again with $\Omega$}. We denote the vertical velocity in the cylinder cross-section as $\mathbf{u}_v^\nu:=(-u_\phi^\nu \sin\phi, u_\phi^\nu \cos\phi)$, so that $\mathbf{u}^\nu:=(\mathbf{u}_v^\nu,u_x^\nu)$, where now $\tan\phi=z/y$ and 
$r=(y^2+z^2)^{1/2}$. Similarly, we write the forcing term $\mathbf{F}=(\mathbf{F_1}, F_2)$ with $\mathbf{F_1}=(-f_1\sin\phi, f_1\cos\phi)$ and $F_2=f_2$. We use the subscript $v$ to denote vertical derivatives, e.g.  $\nabla_v:=(\partial_{y}, \partial_{z})^T$, $\Delta_v:=\partial_{y}^2+\partial_{z}^2$. The cross-sectional diffusion operator 
is given by $L_\nu=-(\nabla_v\cdot J\nabla_v)$, where 
\begin{equation*}
    J=\left [
\begin{pmatrix}
\nu_1\cos^2\phi+\nu_2\sin^2\phi & (\nu_1-\nu_2)\sin\phi \cos\phi\\
(\nu_1-\nu_2)\sin\phi \cos\phi & \nu_1\sin^2\phi+\nu_2\cos^2\phi
\end{pmatrix}
\right ]
\end{equation*}
is a symmetric, time-independent matrix.

 With this notation, ignoring the equation for pressure which is decoupled, \eqref{NSEPIPE} for the velocity
 $\mathbf{u}^\nu$ becomes :
\begin{equation}
    \begin{cases}
    \label{NSEPIPEC}
        \partial_t \mathbf{u}_v^\nu-L_\nu \mathbf{u}_v^\nu=\mathbf{F_1},\quad \quad \quad \quad & \text{in } \Omega\times[0,T],\\
            \partial_t u_x^\nu+(\mathbf{u}_v^\nu\cdot\nabla_v)u_x^\nu- L_\nu u_x^\nu=F_2, &\text{in }  \Omega\times[0,T],\\
        \mathbf{u}^\nu=(-a\sin\phi, a\cos\phi, b)=:(\mathbf{a},b),  & \text{on } \Omega\times\{t=0\},\\
        \mathbf{u}^\nu=0, & \text{on }  \partial \Omega\times[0,T].
    \end{cases}
\end{equation}
 We note that in the isotropic case, when $\nu_1=\nu_2$, the first equation in (\ref{NSEPIPEC}) is simply a heat equation and the vanishing viscosity limit is more straightforward to prove (see e.g. \cite{lopes2008vanishing,BonaWu2002}). 
%\Annacomment{Dobbiamo essere attenti qui, perche' i coefficienti non possono essere costanti vicino all'asse}

 We observe also that the operator $L_\nu$ is uniformly elliptic in $\cC$ with smooth coefficients. As a matter of fact, the eigenvalues of $J$ are $2\nu_1$ and $2\nu_2$ with eigenvectors $\frac{\mathbf{e}_r}{\sqrt{2}}$ and $\frac{-\mathbf{e}_\phi}{\sqrt{2}}$. Hence $L_\nu$ is uniformly elliptic outside any neighborhood of the origin.
 
 %In order to eliminate the problem near the origin, where the operator is not well defined, we need to impose that the system become isotropic $(\nu_1=\nu_2)$ when $r$ goes to $0$. In this way we can extend with continuity the operator in $r=0$. This request is consistent because physically we expect that the direction does not count near the origin. Another option could to consider a domain given by two symmetric cylinders (centered in $x=0$). In this case the results are similar to the pipe parallel symmetry (as in the isotropic case).

 Formally taking the limit $\nu_1, \nu_2 \to 0$ gives:
\begin{equation}
\begin{cases}
\label{EEPIPEC}
    \partial_t\mathbf{u}_v^0=\mathbf{F_1},\quad \quad \quad \quad \quad \quad \quad \quad \quad \quad & \text{in } \Omega\times[0,T],\\
    \partial_tu_x^0+(\mathbf{u}_v^0\cdot\nabla_v)u_x^0=F_2, & \text{in }  \Omega\times[0,T]\\
            \mathbf{u}^0=(-a\sin\phi, a\cos\phi, b)=:(\mathbf{a},b), & \text{on } \Omega\times\{t=0\}.
    \end{cases}
\end{equation}
We impose next the compatibility conditions to prevent the formation of an initial layer:
\begin{equation}
    \begin{aligned}
        \label{CCPIPE}
        &(\mathbf{a},b)=(0,0),  \quad & \text{on }  \partial\Omega\times[0,T],\\
        &(\nabla_v\cdot J\nabla_v)\mathbf{a}+\mathbf{F_1}=0,   \quad & \text{in }  \partial \Omega\times[0,T],\\
        &(\nabla_v\cdot J\nabla_v)b+F_2=0, \quad  & \text{on } \partial \Omega\times[0,T];
    \end{aligned}
\end{equation}
or the equivalent in cylindrical coordinates:
\begin{equation}
    \label {ccpc2}
    \begin{aligned}
        &a(1)=0 \quad b(1,\phi)=0,\\
        &\nu_1(a''(1)+a'(1))-\nu_2a(1)+f_1=0,\\
        &\nu_1(\partial_{rr} b(1,\phi)+\partial_r b(1,\phi))+f_2=0.
    \end{aligned}
\end{equation}
The well posedness of the systems (\ref{NSEPIPEC}) and (\ref{EEPIPEC}) is standard (e.g.  using the methods discussed in \cite{evans2022partial}).
In particular, assuming that the initial data $\mathbf{a}, b$ are in $H^m(\Omega)$ with $m>4$ and $\mathbf{F}\in L^\infty([0,T]\times\Omega)$, there exists a unique solution $\mathbf{u}^\nu$ in the space $L^\infty([0,T];H^1_0(\Omega))$ and $\mathbf{u}^0 \in C^1(0,T;H^m(\Omega))$.
Then, using also Hardy's inequality, it follows that:
\begin{equation*}
    \begin{aligned}
       & u_\phi^0, \partial_ru_\phi^0, \frac{u_\phi^0}{r}, \left(-\frac{1}{r}\partial_r(r\partial_xu_\phi^0)+\frac{u_\phi^0}{r^2}\right) \in L^\infty([0,T]\times\Omega),\\
        &u_x^0, \partial_ru_x^0, \partial_\phi u_x^0, \partial_{r\phi}u_x^0, \left(\frac{1}{r}\partial_r(r\partial_eu_x^0)+\frac{\partial_{\phi\phi} u_x^0}{r^2}\right)\in L^\infty([0,T]\times\Omega).
    \end{aligned}
\end{equation*}

\subsection{Boundary Layer Analysis} \label{ss:PipeLayer}
To handle the mismatch between the boundary conditions of $\mathbf{u}^\nu$ and $\mathbf{u}^0$ in the vanishing viscosity limit, we proceed with studying the boundary layer via correctors. The boundary has two circular components at $r=1,\delta$. Hence, two correctors are needed, which we denote $\bm \theta$ and $\bm \theta_\delta$, respectively, and similarly for the pressure. The corrector analysis is very similar on the component of the boundary. For brevity, we discuss in detail only the corrector at the outher boundary $\bm \theta$.

We assume that the correctors also satisfy the pipe-parallel symmetry. Moreover, as in the channel case, we expect the boundary layer to be of Prandtl type, that is, to be of thickness $\sqrt{\nu_1}$. The rigorous convergence analysis that we will perform will establish {\em a posteriori} the validity of this assumption. 
We then define the zeroth-order correctors $\bm \theta^0$ and $q^0$ near the outer boundary respectively for velocity and pressure:
\begin{align} 
    &\bm \theta^0\left(t,\frac{1-r}{\sqrt{\nu_1}},\phi\right)=\theta_\phi^0\left(t,\frac{1-r}{\sqrt{\nu_1}}\right)\mathbf{e}_\phi+\theta_x^0\left(t,\frac{1-r}{\sqrt{\nu_1}},\phi\right)\mathbf{e}_x, \nonumber \\
    &q^0\left(t,\frac{1-r}{\sqrt{\nu_1}}\right), \nonumber
\end{align}
with $\bm \theta^0$ and $q^0$ depending on $\nu_1$ only through their arguments. Hence, we introduce the stretched variable $Z=\frac{1-z}{\sqrt{\nu_1}}$ and pose the correctors on the domain $\Omega_\infty=[0,+\infty)\times[0,2\pi)$. Finally,  we impose the decay conditions:
\begin{equation} \label{eq:PipeCorrectorDecay}
    \bm \theta^0 \to 0 \quad \textit{as} \quad Z\to \infty, \quad q^0\to0  \quad \textit{as} \quad Z\to\infty.
\end{equation}
Next, we formally define the approximate Navier-Stokes solution near the outer boundary as
\begin{align}
\label{APPU}
    &\mathbf{u}^{app}(t,r,\phi)=\mathbf{u}^{ou}(t,r,\phi)+\bm \theta^0\left(t,\frac{1-r}{\sqrt{\nu_1}},\phi\right), \\
\label{APPP}
    &p^{app}(t,r,)=p^{ou}(t,r)+q^0\left(t,\frac{1-r}{\sqrt{\nu_1}}\right),
\end{align}
where $\mathbf{u}^{ou}$ and $p^{ou}$ are the outer solutions to be determined.
Because of \eqref{eq:PipeCorrectorDecay}, again formally, the approximate solution matches with the outer solution in the zero-viscosity limit.

Using (\ref{APPU}) in (\ref{NSEPIPE}) yields:
\begin{align}
&\partial_tu_\phi^{ou}+\partial_t\theta_\phi^0-\nu_1\frac{1}{r}\partial_r(r\partial_ru_\phi^{ou})+\nu_2\frac{1}{r^2}u_\phi^{ou}-\nu_1\frac{1}{r}\partial_r(r\partial_r\theta_\phi^0)+\nu_2\frac{1}{r^2}\theta_\phi^0 \nonumber \\
& \qquad =\partial_tu_\phi^{ou}+\partial_t\theta_\phi^0-\nu_1\frac{1}{r}\partial_r(r\partial_ru_\phi^{ou})+\nu_2\frac{1}{r^2}u_\phi^{ou}+\sqrt{\nu_1}\frac{1}{r}\partial_Z\theta_\phi^0-\partial_{ZZ}\theta_\phi^0+\nu_2\frac{1}{r^2}\theta_\phi^0=f_1, \nonumber \tag{a} \\
&\partial_tu_x^{ou}+\partial_t\theta_x^0+\frac{1}{r}u_\phi^{ou}\partial_\phi u_x^{ou}+\frac{1}{r}u_\phi^{ou}\partial_\phi\theta_x^0+\frac{1}{r}\theta_\phi^0\partial_\phi u_x^{ou}+\frac{1}{r}\theta_\phi^0\partial_\phi \theta_x^0 \nonumber \\
& \qquad -\nu_1\frac{1}{r}\partial_r(r\partial_ru_x^{ou})-\nu_2\frac{1}{r^2}\partial_{\phi\phi}u_x^{ou}+\sqrt{\nu_1}\frac{1}{r}\partial_Z\theta_x^0-\partial_{ZZ}\theta_x^0-\nu_2\frac{1}{r^2}\partial_{\phi\phi}
\theta_x^0=f_2. \nonumber \tag{b}
\end{align}
By dropping lower-order terms in $\nu_1$ and $\nu_2$ it follows that:
\begin{equation} \label{eq:PipeOuterSol}
        \begin{cases}
            \partial_tu_\phi^{ou}=f_1,\\
            \partial_tu_x^{ou}+\frac{1}{r}u_\phi^{ou}\partial_\phi u_x^{ou}=f_2.
        \end{cases}
\end{equation}
Therefore, $\mathbf{u}^{ou}$ satisfies the same equation of $\mathbf{u}^0$ with the same initial condition and we can assume $\mathbf{u}^{ou}=\mathbf{u}^0$.
The equations for the velocity corrector then reduce to:
    \begin{equation} \label{eq:PipeCorrectorEqs}
       \begin{cases}
                \partial_t\theta_\phi^0-\partial_{ZZ}\theta_\phi^0=0,\\
                \partial_t\theta_x^0+u_\phi^0(t,1)\partial_\phi\theta_x^0+\theta_\phi^0\partial_\phi u_x^0(t,1,\phi)+\theta_\phi^0\partial_\phi \theta_x^0-\partial_{ZZ}\theta_x^0=0,\\
                    (\theta_\phi^0,\theta_x^0)|_{Z=0}=(-u_\phi^0(t,1),-u_x^0(t,1,\phi)),\\
                     (\theta_\phi^0,\theta_x^0)|_{Z=\infty}=(0,0),\\
                      (\theta_\phi^0,\theta_x^0)|_{t=0}=(0,0).
        \end{cases}
    \end{equation}
Again by a standard approach, system  \eqref{eq:PipeCorrectorEqs} is well posed and the estimates derived in Appendix \ref{s:rates} for the corrector apply.
    
Utilizing (\ref{APPU}) and (\ref{APPP}) in the fist equation of (\ref{NSEPIPE}),  we find the system for the pressure:
\begin{equation*}
    \frac{-(u_\phi^0)^2}{r}+\partial_rp^{ou}-\frac{(\theta_\phi^0)^2}{r}-\frac{1}{\nu_1}\partial_Zq^0-\frac{u_\phi^0\theta_\phi^0}{r}=0,
\end{equation*}
Again by dropping the lower order terms in $\nu_1$, we have:
\begin{equation*}
    \frac{-(u_\phi^0)^2}{r}+\partial_rp^{ou}=0,\qquad 
    \partial_Zq^0=0.
\end{equation*}
Hence,  we have that $p^{ou}=p^{0}$ and we can choose $q^0=0$, so that $p^{app}=p^0$. Even if the boundary is curved, because of the symmetry imposed on the flow, the pressure does not have to be corrected at zeroth-order (for a discussion of this point on general domains for linearized flows, see \cite{2018JMFM...20.1405G}).

As for the channel case, one needs to introduce a suitable cut-off so that the corrector acts on the boundary layer in physical variables. To this end, we let $\rho(r)$ be a smooth function such that $\rho(r)$ is 0 in $[0,\frac{1}{4}]$ and 1 in $[\frac{1}{2},1]$. We then redefine the approximate Navier-Stokes solutions as \ $\tilde{\mathbf{u}}^{app}=\tilde{u}_\phi^{app}\mathbf{e}_\phi+\tilde{u}_x^{app}\mathbf{e}_x$, where 
\begin{align}
    &\tilde{u}_\phi^{app}(t,r)=u_\phi^0(t,r)+\rho(r)\theta_\phi^0\left(t,\frac{1-r}{\sqrt{\nu_1}}\right), \nonumber \\
    &\tilde{u}_x^{app}(t,r,\phi)=u_x^0(t,r,\phi)+\rho(r)\theta_x^0\left(t,\frac{1-r}{\sqrt{\nu_1}},\phi\right). \nonumber
\end{align}
 It follows that ${\tilde{\mathbf{u}}}^{app}$ satisfies the following system:
\begin{align}
    &r\partial_rp^{app} -(\tilde{u}_\phi^{app})^2=r\partial_rp^0-(u_\phi^0)^2-(\rho\theta_\phi^0)^2-2\rho u_\phi^0\theta_\phi^0=A,\nonumber \\
    &\partial_t\tilde{u}_\phi^{app}-\nu_1\frac{1}{r}\partial_r(r\partial_r\tilde{u}_\phi^{app})+\nu_2\frac{1}{r^2}\tilde{u}_\phi^{app}=f_1+B+C+D,\nonumber \\
    &\partial_t\tilde{u}_x^{app}+\frac{1}{r}\tilde{u}_\phi^{app}\partial_\phi \tilde{u}_x^{app}-\nu_1\frac{1}{r}\partial_r(r\partial_r\tilde{u}_x^{app})-\nu_2\frac{1}{r^2}\partial_{\phi\phi}\tilde{u}_x^{app}=f_2+E+F+G,
    \nonumber 
\end{align}
where the terms on the right-hand side are given by
\begin{align}
    A&=-(\rho\theta_\phi^0)^2-2\rho u_\phi^0\theta_\phi^0, \qquad \qquad \quad & \text{(a)} \nonumber  \\
    B&=\nu_1\left(-\frac{1}{r}\partial_r(r\partial_ru_\phi^0)-\frac{1}{r}\rho'\theta_\phi^0-\rho''\theta_\phi^0\right), & \text{(b)} \nonumber \\
    C&=\sqrt{\nu_1}\left(\frac{1}{r}\rho\partial_Z\theta_\phi^0+2\rho'\partial_Z\theta_\phi^0\right), & \text{(c)} \nonumber \\
    D&=\nu_2\left(\frac{1}{r^2}u_\phi^0+\frac{1}{r^2}\rho\theta_\phi^0\right),  & \text{(d)} \label{eq:PipeCorrectorForcingDef}\\
    E&=-\nu_1\left(\frac{1}{r}\rho'\theta_x^0+\rho''\theta_x^0+\frac{1}{r}\partial_r(r\partial_ru_x^0)\right), 
    & \text{(e)} \nonumber \\
    F&=\sqrt{\nu_1}\left(\frac{\rho}{r}\partial_Z\theta_x^0+2\rho'\partial_Z\theta_X^0\right), & \text{(f)}\nonumber \\
    G&=-\nu_2\frac{1}{r^2}\partial_{\phi\phi}u_x^0-\nu_2\frac{1}{r^2}\partial_{\phi\phi}\rho\theta_x^0+\rho\left(\frac{\rho}{r}-1\right)\theta_\phi^0\partial_\phi\theta_x^0+\rho\partial_\phi\theta_x^0\left(\frac{u_\phi^0}{r}-u_\phi^0(t,1)\right) &\;  \nonumber\\
    &\qquad \qquad \qquad +\rho\theta_\phi^0\left(\frac{1}{r}\partial
    _\phi u_x^0-\partial_\phi u_x^0(t,1,\phi)\right),  &\text{(g)} \nonumber
\end{align}
and primes denote the derivative with respect to the argument.
In addition, the approximate solution satisfies the following initial and boundary conditions:
\begin{equation*}
    \tilde{\mathbf{u}}^{app}|_{t=0}=\mathbf{u}_0, \qquad 
    \tilde{\mathbf{u}}^{app}|_{r=1}=0.
\end{equation*}
Next, we define the approximation error \ $\mathbf{u}^{err}(t,r,\phi):=\mathbf{u}^\nu(t,r,\phi)-\tilde{\mathbf{u}}^{app}(t,r,\phi)$, which satisfies the following system:
\begin{align}
    \label{errprimo}
    &(u_\phi^{err})^2+2u_\phi^{err}\tilde{u}_\phi^{app}-r\partial_rp^{err}=A, \\
\label{errdue}
    &\partial_tu_\phi^{err}-\frac{\nu_1}{r}\partial_r(r\partial_ru_\phi^{err})+\frac{\nu_2}{r^2}u_\phi^{err}=-B-C-D, \\
\label{errtre}
    &\partial_tu_x^{err}+\frac{u_\phi}{r}\partial_\phi u_x^{err}+ \frac{u_\phi^{err}}{r}\partial_\phi \tilde{u}_x^{app}-\nu_1\frac{1}{r}\partial_r(r\partial_ru_x^{err})-\nu_2\frac{1}{r^2}\partial_{\phi\phi}u_x^{err}=-E-F-G, 
\end{align}
with initial and boundary conditions:
\begin{equation*}
    \mathbf{u}^{err}|_{r=1}=0,  \quad \quad \mathbf{u}^{err}|_{t=0}=0.
\end{equation*}
%\Annacomment{Qui sopra bisognerebbe controllare che non portiamo mai fuori i coefficienti di viscosita' dalle derivate, perche' non sono piu' costanti. Anche nella Proposizione \ref{prop2}}
\subsection{Convergence rates} \label{ss:PipeEnergy}

In this subsection, we tackle the vanishing viscosity limit by estimating $\bu^{err}$ in various norms.
We start with the following proposition.

\begin{proposition}
\label{prop2}
Let $A, B, C, D, E, F, G$ be  as in \eqref{eq:PipeCorrectorForcingDef}. Then, there exists a constant $c>0$, independent of $\nu_1$, $\nu_2$, such that 
\begin{align}
    &\left\|\frac{A}{r^2}\right\|_{L^\infty(0,T;L^1(\Omega))}\leq c\nu_1^\frac{1}{2}, \qquad 
     \left\|\frac{A}{r}\right\|_{L^\infty(0,T;L^2(\Omega))}\leq c\nu_1^\frac{1}{4}, \nonumber \\
    &||B+C+D||_{L^\infty(0,T;L^2(\Omega))} \leq c(\nu_1^\frac{3}{4}+\nu_2), \label{eq:prop2}\\
    &||B+C+D||_{L^\infty((0,T)\times\Omega)} \leq c(\nu_1^\frac{1}{2}+\nu_2), \nonumber \\
    &\normone{E+F+G}\leq c(\nu_1^\frac{3}{4}+\nu_2), \nonumber \\
    &\normone{\partial_\phi(E+F+G)}\leq c(\nu_1^\frac{3}{4}+\nu_2). \nonumber
\end{align}
\end{proposition}

\begin{proof} 
We recall that the measure  put on $\Omega$ is  $d\Omega=rdrd\phi$.  Using the explicit formulas for $A, B, C, D$ in \eqref{eq:PipeCorrectorForcingDef}, we have that
\begin{align}
    &\left\|\frac{A}{r^2}\right\|_{L^1(\Omega)}\leq \int_\frac{1}{4}^1\frac{(\rho\theta_\phi^0)^2+2|\rho u_\phi^0\theta_\phi^0|}{r^2}rdr \nonumber \\
    & \qquad \qquad  \qquad \leq c(1+||u_\phi^0||_{L^\infty(\Omega)})\int_0^\infty[(\theta_\phi^0(t,Z)^2+|\theta_\phi^0(t,Z)|]dZ\leq c\sqrt{\nu_1} \nonumber \\
    &\left\|\frac{A}{r}\right\|_{L^2(\Omega)}^2\leq \int_\frac{1}{4}^1\frac{(\rho\theta_\phi^0)^4+4|\rho  u_\phi^0\theta_\phi^0|^2}{r^2}rdr \nonumber \\
    & \qquad \qquad  \qquad  \leq c(1+||(u_\phi^0)^2||_{L^\infty(\Omega)})\int_0^\infty((\theta_\phi^0(t,Z)^4+(\theta_\phi^0(t,Z))^2)dZ\leq c\sqrt{\nu_1}, \nonumber 
\end{align}
\begin{align}
    &||B||_{L^2(\Omega)}^2\leq \int_\Omega\nu_1^2(\Delta u_\phi^0)^2+c\int_\frac{1}{4}^1\nu_1^2(\theta_\phi^0)^2dr\leq \nu_1^2||\Delta u_\phi^0||_{L^2(\Omega)}^2+c\nu_1^\frac{5}{2}||\theta_\phi^0||_{L^2(0,\infty)}^2, \nonumber \\
    &||C||_{L^2(\Omega)}^2\leq \int_\frac{1}{4}^1\nu_1\left(\frac{\partial_Z\theta_\phi^0}{r}+2\partial_Z\theta_\phi^0\right)^2rdr\leq\nu_1^\frac{3}{2}||\partial_Z\theta_\phi^0||_{L^2(0,\infty)}^2, \nonumber \\
    &||D||_{L^2(\Omega)}^2\leq \int_\Omega\nu_2^2\frac{(u_\phi^0)^2}{r^4}rdr+\int_\frac{1}{4}^1\nu_2^2\frac{(\theta_\phi^0)^2}{r^4}rdr\leq \nu_2^2||\Delta(u_\phi^0e_\phi)||_{L^2(\Omega)}^2+\nu_2^\frac{5}{2}||\theta_\phi^0||_{L^2(0,\infty)}^2 
    \nonumber \\
    &||E||_{L^2(\Omega)}^2\leq c\nu_1^2\int_\Omega(\Delta u_x^0)^2+c\nu_1^\frac{5}{2}\left(\int_\frac{1}{4}^1(\theta_x^0)^2+\int_\frac{1}{4}^1(\partial_{\phi\phi}\theta_x^0)^2\right)\leq c\nu_1^2, \nonumber \\
    &||F||_{L^2(\Omega)}^2\leq c\nu_1\int_\frac{1}{4}^14\frac{(\partial_Z\theta_x^0)^2}{r^2}+(\partial_Z\theta_x^0)^2\leq c\nu_1^\frac{3}{2}||\partial_Z\theta_x^0||^2_{L^2(\Omega_\infty)}\leq c\nu_1^\frac{3}{2},\nonumber
\end{align}
while for $G$ we have that
\begin{align}
    &||G||_{L^2(\Omega)}\leq c\nu_2(||\Delta u_x^0||_{L^2(\Omega)}+||\Delta\theta_x^0||_{L^2(\Omega_\infty)})+c\nu_1^\frac{3}{4}||\theta_\phi^0||_{L^\infty(0,\infty)}||\langle Z \rangle \partial_\phi \theta_x^0||_{L^2(\Omega_\infty)} \nonumber \\
    & \qquad \qquad + c\nu_1^\frac{3}{4}(||u_\phi^0||_{L^\infty(\Omega)}
    +||\partial_ru_\phi^0||_{L^\infty(\Omega)})||\langle Z \rangle \partial_\phi \theta_x^0||_{L^2(\Omega_\infty)}+c\nu_1^\frac{3}{4}(||\partial_\phi u_x^0||_{L^\infty(\Omega)}+ \nonumber \\
    &\qquad \qquad  \qquad ||\partial_{r\phi}u_x^0||_{L^\infty(\Omega)})||\langle Z \rangle \theta_\phi ^0||_{L^2(0,\infty)}. 
     \nonumber 
\end{align}
We decompose the expression above as follows:
\begin{equation*}
        G=I_1+I_2+I_3+I_4,
\end{equation*}
where 
\begin{align}
  I_1&=\frac{-\nu_2}{r^2}(\partial_{\phi\phi}u_x^0+\partial_{\phi\phi}\rho\theta_x^0), \nonumber \\
  I_2&=\rho\left(\frac{\rho}{r}-1\right)\theta_\phi^0\partial_\phi\theta_x^0, \nonumber \\
  I_3&=\rho\left(\frac{u_\phi^0(t,r)}{r}-u_\phi(t,1)\right)\partial_\phi\theta_x^0, \nonumber \\
  I_4&=\rho\left(\frac{\partial_\phi u_x^0(t,r,\phi)}{r}-\partial_\phi u_x(t,1,\phi)\right)\theta_\phi^0.
  \nonumber 
\end{align}
We can estimate each term on the right-hand side as follows:
\begin{align}
    &||I_1||_{L^2(\Omega)}\leq \nu_2(||\Delta u_x^0||_{L^2(\Omega)}+||\Delta \theta_x^0||_{L^2(\Omega)}),
    \nonumber \\
    &||I_2||_{L^2(\Omega)}^2=\int_\Omega\rho^2\left(\frac{\rho}{r}-1\right)^2(\theta_\phi^0)^2(\partial_\phi\theta_x^0)^2rdrd\phi\leq c\int_0^{2\pi}\Big(\int_\frac{1}{4}^\frac{1}{2}(\theta_\phi^0)^2(\partial_\phi\theta_x^0)^2rdr \nonumber \\
    & \qquad \qquad +\int_\frac{1}{2}^1(r-1)^2(\theta_\phi^0)^2(\partial_\phi\theta_x^0)^2rdr\Big)d\phi \leq c\int_0^{2\pi}\Big(\int_\frac{1}{2\sqrt{\nu_1}}^\frac{3}{4\sqrt{\nu_1}}\frac{\langle Z \rangle ^2}{\langle Z \rangle ^2}\nu_1^\frac{1}{2}(\theta_\phi^0\partial_\phi\theta_x^0)^2rdZ \nonumber \\
    & \qquad \qquad \qquad +\int_{0}^{\frac{1}{2\sqrt{\nu_1}}}\nu_1^\frac{3}{2}(\partial_\phi^0\partial_\phi\theta_x^0)^2Z^2rdZ\Big)d\phi\leq c\nu_1^\frac{3}{2}||\theta_\phi^0||_{L^\infty(0,\infty)}^2||\langle Z \rangle \partial_\phi\theta_x^0||_{L^2(\Omega_\infty)}^2, \nonumber \\
    &||I_3||_{L^2(\Omega)}^2=\int_0^{2\pi}\int_0^1\Big(\rho(\frac{u_\phi^0(t,r)}{r}-u_\phi(t,1))\partial_\phi\theta_x^0\Big)^2rdrd\phi=\int_0^{2\pi}\int_0^1\Big(\rho\frac{r-1}{r}(\partial_ru_\phi(t,\xi) \nonumber \\
    &\qquad \qquad - u_\phi^0(t,1))\partial_\phi\theta_x^0\Big)^2rdrd\phi\leq c\nu_1^\frac{3}{2}(||u_\phi^0||_{L^\infty(\Omega)}+||\partial_ru_\phi^0||_{L^\infty(\Omega)})^2||\langle Z \rangle \partial_\phi\theta_x^0||_{L^2(\Omega)}^2, \nonumber \\
    &||I_4||_{L^2(\Omega)}^2=\int_0^{2\pi}\int_0^1\Big(\rho(\frac{\partial_\phi u_x^0(t,r,\phi)}{r}-\partial_\phi u_x(t,1,\phi))\theta_\phi^0\Big)^2rdrd\phi \nonumber \\
    &\qquad \qquad =\int_0^{2\pi}\int_0^1\Big(\rho\frac{r-1}{r}(\partial_r\partial_\phi u_x(t,\xi,\phi)-\partial_\phi u_\phi^0(t,1,\phi))\theta_\phi^0\Big)^2rdrd\phi \nonumber \\
    & \qquad \qquad \qquad \leq c\nu_1^\frac{3}{2}(||\partial_\phi u_x^0||_{L^\infty(\Omega)}+||\partial_r\partial_\phi u_x^0||_{L^\infty(\Omega)})^2||\langle Z \rangle \theta_\phi^0||_{L^2(0,\infty)}^2. \nonumber 
 \end{align}
Combining the bounds above gives the desired estimates, noting also that differentiating with respect to $\phi$ gives similar estimates, which concludes the proof.
\end{proof}

We now turn to the proof of our main convergence result for pipe flows.

\begin{theorem}
\label{teo6}
Assume the initial data $\mathbf{a}, b$ in $H^m(\Omega)$, $m>4$, the forcing $\mathbf{F}\in C^\infty([0,T]\times\Omega)$, and assume the compatibility conditions (\ref{ccpc2}).  There exists a constant $c$ independent of $\nu_1$ and $\nu_2$ such that 
    \begin{equation} \label{eq:PipeVVLConvergenceRates}
        \normone{\mathbf{u}^\nu-\mathbf{u}^0-\bm \theta^0}\leq c\, (\nu_1^\frac{3}{4}+\nu_2).
    \end{equation}
\end{theorem}

The following corollary is a direct consequence of the triangle inequality and the estimate on  $\normone{\bm \theta^0}$ from Appendix \ref{s:rates}. It establishes the vanishing viscosity limit for any value of $\nu_1$ and $\nu_2$. However, we will be able to obtain  convergence in higher norms only in certain regimes, depending on the relative size of $\nu_1$ and $\nu_2$.

\begin{corollary} \label{cor:PipeVVL}
    Under the hypotheses of Theorem \ref{teo6}, the following estimate holds:
    \begin{equation} \label{eq:PipeVVL}
        \normone{\mathbf{u}^\nu-\mathbf{u}^0}\leq c(\nu_1^\frac{1}{4}+\nu_2).
    \end{equation}
\end{corollary}

%\Annacomment{Controllare la dimostrazone, possiamo anche devidere di fare tutto in un anello con coefficienti costanti}

\begin{proof}[Proof of Theorem \ref{teo6}]
We begin by multiplying (\ref{errdue}) by $u_\phi^{err}$ and integrating  over $\Omega$ to obtain:
\begin{align}
\frac{1}{2}\frac{d}{dt}||u_\phi^{err}||_{L^2(\Omega)}^2-\int_\Omega\frac{\nu_1}{r}\partial_r(r\partial_ru_\phi^{err})u_\phi^{err}rdrd\phi+\int_\Omega\frac{\nu_2}{r^2}u_\phi^{err}u_\phi^{err}=-\int_\Omega(B+C+D)u_\phi^{err}, \nonumber
\end{align}
where the terms $B,\,C,\,D$ are as in \eqref{eq:PipeCorrectorForcingDef}.
Integrating the second term by parts and using H\"older and Young's inequalities gives:
\begin{align}
    &\frac{1}{2}\frac{d}{dt}||u_\phi^{err}||_{L^2(\Omega)}^2+\int_\Omega\nu_1(\partial_ru_\phi^{err})^2rdrd\phi+
    \int_{\Omega} 
    \nu_2\left\|\frac{u_\phi^{err}}{r}\right\|_{L^2(\Omega)}^2 
    \nonumber \\
    & \leq  c\,||B+C+D||_{L^2(\Omega)}^2+c \,||u_\phi^{err}||_{L^2(\Omega)}^2 .
\end{align}
Next, we integrate in time and apply Gr\"onwall's inequality:
\begin{equation}
\label{uphierr}
\begin{aligned}
        &||u_\phi^{err}||_{L^\infty(0,T; L^2(\Omega))}+\sqrt{\nu_1}||\partial_ru_\phi^{err}||_{L^2(0,T;L^2(\Omega))}+\sqrt{\nu_2}\left\|\frac{u_\phi^{err}}{r}\right\|_{L^2(0,T;L^2(\Omega))}\\
        &\qquad \qquad \leq c\, ||B+C+D||_{L^2(0,T;L^2(\Omega))}\leq c\,(\nu_1^\frac{3}{4}+\nu_2).
        \end{aligned}
\end{equation}
% We introduce now different ideas to study the convergence in higher norms.
 Similarly, multiplying (\ref{errtre}) by $u_x^{err}$ and integrating over $\Omega$ gives:
\begin{equation*}
\begin{aligned}
    &\frac{1}{2}\frac{d}{dt}||u_x^{err}||_{L^2(\Omega)}^2+\int_\Omega\frac{u_\phi}{r}\partial_\phi u_x^{err}u_x^{err}+\int_\Omega\frac{u_\phi^{err}}{r}\partial_\phi \tilde{u}_x^{app}u_x^{err}-\nu_1\int_\Omega\frac{1}{r}\partial_r(r\partial_ru_x^{err})u_x^{err}rdrd\phi\\
    & \qquad \qquad \nu_2\int_{\Omega}\frac{1}{r^2}\partial_{\phi\phi}u_x^{err}u_x^{err}rdrd\phi=-\int_\Omega(E+F+G)u_x^{err}.
    \end{aligned}
\end{equation*}
We integrate the fourth and the fifth term on the left-hand side by parts and observe that the second integral is equal to 0:
\begin{equation*}
\begin{aligned}
    &\frac{1}{2}\frac{d}{dt}||u_x^{err}||_{L^2(\Omega)}^2+\int_\Omega\frac{u_\phi^{err}}{r}\partial_\phi \tilde{u}_x^{app}u_x^{err}+\nu_1\int_\Omega (\partial_ru_x^{err})^2rdrd_\phi\\
    & \qquad \qquad +\nu_2\int_{\Omega}\frac{1}{r^2}(\partial_{\phi}u_x^{err})^2rdrd\phi=-\int_\Omega(E+F+G)u_x^{err}.
    \end{aligned}
\end{equation*}
It follows that:
\begin{equation*}
\begin{aligned}
    &\frac{1}{2}\frac{d}{dt}||u_x^{err}||_{L^2(\Omega)}^2+\nu_1|| \partial_ru_x^{err}||^2_{L^2(\Omega)}+\nu_2||\frac{1}{r}(\partial_{\phi}u_x^{err})||_{L^2(\Omega)}^2\\
    &\leq\int_\Omega(E+F+G)u_x^{err}-\int_{\Omega}\frac{u_\phi^{err}}{r}\partial_\phi\tilde{u}_x^{app}u_x^{err}rdrd\phi\\
    &\leq \, ||E+F+G||_{L^2(\Omega)}||u_x^{err}||_{L^2(\Omega)}+||u_\phi^{err}||_{L^2(\Omega)}(||\partial_\phi u_x^0||_{L^\infty(\Omega)}+||\partial_\phi\theta_x^0||_{L^\infty(\Omega_\infty)})||u_x^{err}||_{L^2(\Omega)}.
    \end{aligned}
\end{equation*}
Again, integrating in  time and applying Young and Gr\"onwall's inequalities yields:
\begin{equation}  \label{uxerr}
\begin{aligned}
       &||u_x^{err}||_{L^\infty(0,T;L^2(\Omega))}+\sqrt{\nu_1}|| \partial_ru_x^{err}||_{L^2(0,T;L^2(\Omega))}+\sqrt{\nu_2}\left\|\frac{1}{r}(\partial_{\phi}u_x^{err})\right\|_{L^2(0,T;L^2(\Omega))}\\
        &\qquad \qquad \leq ||E+F+G||_{L^2(0,T;L^2(\Omega))}+||u_\phi^{err}||_{L^2(0,T;L^2(\Omega))}\leq c\,(\nu_1^\frac{3}{4}+\nu_2),
\end{aligned}
\end{equation}
where we used Proposition \ref{prop2}.
Finally, we obtain \eqref{eq:PipeVVLConvergenceRates} by combining \eqref{uphierr} and \eqref{uxerr}
\end{proof}

To obtain higher-order estimates $\mathbf{u}^{err}$, we switch to Cartesian coordinates and prove the following convergence result.

\begin{theorem}
\label{teo7} 
    If we consider the initial data $\mathbf{a}, b$ in $H^m(\Omega)$, $m>4$, $\mathbf{F}\in C^\infty([0,T]\times\Omega)$  and the compatibility conditions (\ref{CC1}), we have, for some constant c independent of $\nu_1$ and $\nu_2$:
    \begin{equation*}
    \begin{aligned}
        &\normone{\mathbf{u}^{err}} \leq c(\nu_1^\frac{3}{4}+\nu_2),\\
        &\normthree{\mathbf{u}^{err}} \leq c(\nu_1^\frac{1}{2}+\nu_2).
        \end{aligned}
    \end{equation*}
    If $\nu_1\leq \nu_2$ and $\frac{\nu_2}{\sqrt{\nu_1}}\to 0$:
    \begin{equation*}
            \normtwo{\mathbf{u}^{err}}\leq c\left(\nu_1^\frac{1}{4}+\frac{\nu_2}{\sqrt{\nu_1}}\right).
    \end{equation*}
    If $\nu_2<\nu_1$ and $\frac{\nu_1^\frac{3}{4}}{\sqrt{\nu}}\to 0$:
    \begin{equation*}
    \normtwo{\mathbf{u}^{err}}\leq c\left(\frac{\nu_1^\frac{3}{4}}{\sqrt{\nu_2}}+\sqrt{\nu_2}\right).
\end{equation*}
\end{theorem}
 
 Assuming the same restrictions on the viscosity coefficients, one can obtain the convergence of the pressure. 

\begin{proof}
First we write the equations (\ref{errdue}) and (\ref{errtre}) in these coordinates:
\begin{align}
    \label{errcc1}
    &\partial_t\mathbf{u}_v^{err}-(\nabla_v\cdot J\nabla_v)\mathbf{u}_v^{err}=\mathbf{g}_2, \\
    \label{errcc2}
    &\partial_tu_x^{err}+(\mathbf{u}_v^\nu \cdot \nabla_v)u_x^{err}-(\nabla_v \cdot J\nabla_v)u_x^{err}=g_3,
\end{align}
with $\mathbf{u}_v^{err}=(-u_\phi^{err}\sin\phi, u_\phi^{err}\cos\phi)$ and 
where $\mathbf{g}_2=-(B+C+D)\left(-\frac{x_1}{\sqrt{x_1^2+x_2^2}},-\frac{x_2}{\sqrt{x_1^2+x_2^2}}\right)$ and $g_3=-(E+F+G)-(\mathbf{u}_v^{err}\cdot \nabla_v)\tilde{u}_x^{app}$ with homogeneous initial and boundary conditions.

 We notice also that:
\begin{equation*}
    (\mathbf{u}_v^{err}\cdot \nabla_v)\tilde{u}_x^{app}=\frac{u_\phi^{err}}{r}\partial_\phi\tilde{u}_x^{app}.
\end{equation*}
Hence,  the decay properties of the corrector and  the regularity of Euler solution imply that
\begin{equation*}
    \normthree{\frac{1}{r}\partial_\phi\tilde{u}_x^{app}}\leq c.
\end{equation*}
Next, we multiply (\ref{errcc1}) by $\mathbf{u}_v^{err}$ (component by component) and integrate over $\Omega$:
\begin{equation*}
    \frac{1}{2}\frac{d}{dt}||u_v^{err}||_{L^2(\Omega)}^2+\int_\Omega(J\nabla_v u_v^{err}\cdot \nabla_v u_v^{err})=\int_\Omega gu_v^{err}\leq \, ||B+C+D||_{L^2(\Omega)}||u_v^{err}||_{L^2(\Omega)},
\end{equation*}
where we integrated by parts in the second term.
Integrating over time and using Young and Grönwall's inequality gives:
\begin{equation*}
    ||u_v^{err}||_{L^\infty(0,T;L^2(\Omega))}+\sup_{t\in[0,T]}\left(\int_\Omega(J\nabla_v u_v^{err}\cdot \nabla_v u_v^{err})\right)^\frac{1}{2}\leq c\, (\nu_1^\frac{3}{4}+\nu_2),
\end{equation*}
and thus $\normone{\mathbf{u}_v^{err}}\leq c(\nu_1^\frac{3}{4}+\nu_2)$.

 We multiply (\ref{errcc1}) by $-(\nabla_v,J\nabla_v)\mathbf{u}_v^{err}$ (component by component) to obtain:
\begin{equation*}
\begin{aligned}
    &\frac{1}{2}\frac{d}{dt}\int_\Omega(\nabla_vu_v^{err}\cdot J\nabla_vu_v^{err})+||(\nabla_v\cdot J\nabla_v)u_v^{err}||_{L^2(\Omega)}^2=-\int_\Omega g(\nabla_v\cdot J\nabla_v)u_v^{err}\\
    &\leq||B+C+D||_{L^2(\Omega)}||(\nabla_v\cdot J\nabla_v)u_v^{err}||_{L^2(\Omega)}\leq c(\nu_1^\frac{3}{4}+\nu_2)||(\nabla_v\cdot J\nabla_v)u_v^{err}||_{L^2(\Omega)}\\
    &\leq \frac{c^2(\nu_1^\frac{3}{4}+\nu_2)^2}{4}+||(\nabla_v\cdot J\nabla_v)u_v^{err}||_{L^2(\Omega)}^2,
    \end{aligned}
\end{equation*}
where we used Young's inequality. We then integrate in time and apply  Gr\"onwall's inequality:
\begin{equation}
\label{autoval}
    \sup_{t\in[0,T]}\left(\int_\Omega(\nabla_vu_v^{err}\cdot J\nabla_vu_v^{err})\right)^\frac{1}{2}\leq c\,(\nu_1^\frac{3}{4}+\nu_2).
\end{equation}
We recall that $J$ has two eigenvalues: $2\nu_1$ and $2\nu_2$ with respective eigenvectors $\frac{\mathbf{e}_r}{\sqrt{2}}$ and $\frac{-\mathbf{e}_\phi}{\sqrt{2}}$ and that $\min(2\nu_1,2\nu_2)||u||\leq ||(u\cdot Ju)||\leq \max(2\nu_1,2\nu_2)||u||$. 

 If $\nu_1\leq \nu_2$, from (\ref{autoval}) we have:
\begin{equation*}
    \sqrt{\nu_1}\normone{\nabla_v\mathbf{u}_v^{err}}\leq c\, (\nu_1^\frac{3}{4}+\nu_2),
\end{equation*}
and dividing by $\sqrt{\nu_1}$:
\begin{equation*}
    \normone{\nabla_v\mathbf{u}_v^{err}}\leq c\, \left(\nu_1^\frac{1}{4}+\frac{\nu_2}{\sqrt{\nu_1}}\right). 
\end{equation*}
Consequently, we have also to impose that: $\frac{\nu_2}{\sqrt{\nu_1}}\to 0$.

 If $\nu_2< \nu_1$, from (\ref{autoval}) we have:
\begin{equation*}
    \sqrt{\nu_2}\normone{\nabla_v\mathbf{u}_v^{err}}\leq c\, (\nu_1^\frac{3}{4}+\nu_2),
\end{equation*}
and dividing by $\sqrt{\nu_2}$ gives:
\begin{equation*}
    \normone{\nabla_v\mathbf{u}_v^{err}}\leq c\, \left(\frac{\nu_1^\frac{3}{4}}{\sqrt{\nu_2}}+\sqrt{\nu_2}\right).
\end{equation*}
Consequently,  we have also to impose that: $\frac{\nu_1^\frac{3}{4}}{\sqrt{\nu_2}}\to 0$.

 We proceed by studying the equation (\ref{errcc2}). Multiplying (\ref{errcc2}) by $u_x^{err}$ and integrating over $\Omega$:
\begin{equation*}
\begin{aligned}
    &\frac{1}{2}\frac{d}{dt}||u_x^{err}||_{L^2(\Omega)}^2+\int_\Omega(\mathbf{u}_v^\nu\cdot\nabla_v)u_x^{err}u_x^{err}+\int_\Omega(\nabla_vu_x^{err}\cdot J\nabla_vu_x^{err})\\
    &\qquad \qquad \leq \, ||E+F+G||_{L^2(\Omega)}||u_x^{err}||_{L^2(\Omega)}+\left|\int_\Omega (\mathbf{u}_v^{err}\cdot\nabla_v)\tilde{u}_x^{app}u_x^{err}\right|,
    \end{aligned}
\end{equation*}
where the second term on the left-hand side is equal to $0$, because $(\mathbf{u}_v^\nu\cdot\nabla_v)u_x^{err}=\frac{u_\phi^\nu}{r}\partial_\phi u_x^{err}$ and $u_\phi^\nu$ is independent from $\phi$.

 Next, we observe that:
\begin{equation*}
    \left|\int_\Omega (\mathbf{u}_v^{err}\cdot\nabla_v)\tilde{u}_x^{app}u_x^{err}\right|\leq \left\|\frac{1}{r}\partial_\phi \tilde{u}_x^{app}\right\|_{L^\infty(\Omega)}||u_\phi^{err}||_{L^2(\Omega)}||u_x^{err}||_{L^2(\Omega)}\leq c\, (\nu_1^\frac{3}{4}+\nu_2)||u_x^{err}||_{L^2(\Omega)}.
\end{equation*}
Then integrating over time and applying Young and Gr\"onwall's inequalities gives:
\begin{equation*}
    \normone{u_x^{err}}+\sup_{t\in[0,T]} \left(\int_\Omega(\nabla_vu_x^{err}\cdot J\nabla_v u_x^{err})\right)^{1/2}
    \leq c\, (\nu_1^\frac{3}{4}+\nu_2).
\end{equation*}
By multiplying (\ref{errcc2}) by $-(\nabla_v\cdot J\nabla_v)u_x^{err}$ and integrating over $\Omega$, it follows that 
\begin{equation}
\begin{aligned}
\label{stimaderivate}
    &\frac{1}{2}\frac{d}{dt}\int_\Omega(\nabla_vu_x^{err}\cdot J\nabla_vu_x^{err})+||(\nabla_v\cdot J\nabla_v)u_x^{err}||_{L^2(\Omega)}^2-\int_\Omega(\mathbf{u}_v^\nu\cdot\nabla_v)u_x^{err}(\nabla_v\cdot J\nabla_v)u_x^{err}\\
    & \qquad \qquad \leq ||E+F+G||_{L^2(\Omega)}||(\nabla_v\cdot J\nabla_v)u_x^{err}||_{L^2(\Omega)}\\
    &\qquad \qquad \qquad +||u_\phi^{err}||_{L^2(\Omega)}\left\|\frac{1}{r}\partial_\phi\tilde{u}_x^{app}\right\|_{L^\infty(\Omega)}||(\nabla_v\cdot J\nabla_v)u_x^{err}||_{L^2(\Omega)}\\
    &\qquad \leq c\, (\nu_1^\frac{3}{4}+\nu_2)||(\nabla_v\cdot J\nabla_v)u_x^{err}||_{L^2(\Omega)}\leq \frac{c^2\, (\nu_1^\frac{3}{4}+\nu_2)^2}{2}+\frac{1}{2}||(\nabla_v\cdot J\nabla_v)u_x^{err}||_{L^2(\Omega)}^2,
    \end{aligned}
\end{equation}
where:
\begin{equation*}
\begin{aligned}
&\Big|\int_\Omega(\mathbf{u}_v^\nu\cdot\nabla_v)u_x^{err}(\nabla_v\cdot J\nabla_v)u_x^{err}\Big|\leq ||u_v^\nu||_{L^\infty(\Omega)}\left\|\frac{1}{r}\partial_{\phi}u_x^{err}\right\|_{L^2(\Omega)}||(\nabla_v\cdot J\nabla_v)u_x^{err}||_{L^2(\Omega)}\\
&\qquad \qquad \qquad \leq \frac{c}{2}\int_\Omega(\nabla_vu_x^{err}\cdot J\nabla_vu_x^{err})+\frac{1}{2}||(\nabla_v\cdot J\nabla_v)u_x^{err}||_{L^2(\Omega)}^2.
    \end{aligned}
\end{equation*}
We employ the above estimate in (\ref{stimaderivate}):
\begin{equation*}
\begin{aligned}
    &\frac{1}{2}\frac{d}{dt}\int_\Omega(\nabla_vu_x^{err}\cdot J\nabla_vu_x^{err})+||(\nabla_v\cdot J\nabla_v)u_x^{err}||_{L^2(\Omega)}^2\leq c(\nu_1^\frac{3}{4}+\nu_2)^2+||(\nabla_v\cdot J\nabla_v)u_x^{err}||_{L^2(\Omega)}^2\\
    &\qquad \qquad +c \int_\Omega(\nabla_vu_x^{err}\cdot J\nabla_vu_v^{err}),
    \end{aligned}
\end{equation*}
applying again Young's inequality and exploiting the regularity of $\mathbf{u}_v^\nu$.
Finally, we integrate in time  and use Gr\"onwall's inequality once again:
\begin{equation*}
    \sup_{t\in[0,T]} \left(\int_\Omega(\nabla_vu_x^{err}\cdot J\nabla_vu_x^{err})\right)^{1/2} \leq c\, (\nu_1^\frac{3}{4}+\nu_2).
\end{equation*}
We can then  conclude imposing the same constraint of the viscosity coefficients as  done for $\mathbf{u}_\nu^{err}$.
\end{proof}

\section{Semigroup analysis for planar  circularly symmetric case} \label{s:semigroup}
%vedi se citare meglio i 2/3 risultati che usi
We continue our analysis of symmetric solutions of the anisotropic Navier-Stokes equations when viscosity vanishes. In particular we show how to establish the vanishing viscosity limit using semigroup methods. (we refer the reader to e.g. \cite{pazy2012semigroups,kato2013perturbation,LunardiBook1995} for an introduction to semigroups.) 
These methods are well suited to semilinear equations, such as the Navier-Stokes equations, and can be successfully applied to singular limits, as the zero-viscosity limit, especially when the systems at hand are weakly nonlinear. This is the case of symmetric flows we are considering (see \cite{mazzucato2008vanishing,GIE20191237,mazzucato2010vanishing} for the isotropic setting). For brevity, we only discuss the case of planar circularly symmetric flows. The advantage of using semigroup techniques is that no correctors are introduced and no hypotheses are made {\em a priori} on the behavior of the flow in the boundary layer. As a matter of fact, we can treat initial data only in $L^2(\Omega)$. We can also estimate the growth of the derivatives due to the discrepancy between the boundary conditions of the Navier-Stokes and Euler equations. For simplicity, in this analysis we take the forcing term $\mathbf{f}=0$.

Circularly symmetric flows are special cases of pipe parallel flows when the flow is planar, that is, the velocity has the form
\begin{equation}
\label{cs}
    \mathbf{u}(t,r,\phi)=\mathbf{u}_\phi(t,r)\mathbf{e}_\phi
\end{equation}
in the domain $\Omega=\{ (r,\phi)\, | \,  \delta\leq r \leq 1, \phi\in[0,2\pi] \}$. Any time-independent circularly symmetric flow is a steady solution of Euler. Hence, the vanishing viscosity limit reduces to convergence to the initial data. 

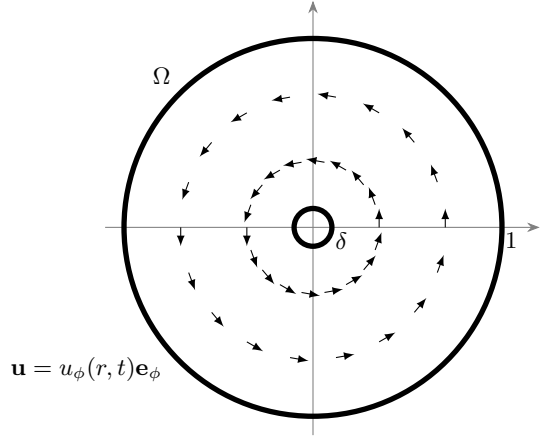
\begin{figure}[h!]
\begin{center}
\begin{tikzpicture}[scale=2.5, >=Stealth]

 \useasboundingbox (-1.2,-1.2) rectangle (1.2,1.2);
% Assi cartesiani
\draw[->, gray] (-1.1,0) -- (1.2,0) node[right] {\footnotesize $ $};
\draw[->, gray] (0,-1.1) -- (0,1.2) node[above] {\footnotesize $ $};

% Disco esterno (bordo)
\draw[line width= 2] (0,0) circle(1);

% Disco interno tratteggiato per boundary layer
%\draw[dashed] (0,0) circle(0.92);

\draw[line width= 2] (0,0) circle (0.1);

% Frecce tangenti corrette (campo u(r) e_phi)
\foreach \r in {0.35, 0.7} {
    \foreach \angle in {0,20,...,340} {
        \draw[-{latex}]
          ({\r*cos(\angle)}, {\r*sin(\angle)})
          -- +({-sin(\angle)*0.1}, {cos(\angle)*0.1});
    }
}

% Etichetta per il boundary layer
%\draw[-{latex}] (0.73,0.73) -- (0.92*0.707, 0.92*0.707);
%\node[right] at (0.7,0.7) {\footnotesize Boundary layer region (near $r = 1$)};

% Coordinata r=1
\node at (1.05, -0.07) {\footnotesize $1$};
\node at (0.15, -0.07) {\footnotesize $\delta$};
\node at (-0.8, 0.8) {\footnotesize $\Omega$};
\node at (-1.2, -0.75) {\footnotesize $\mathbf{u}=u_\phi(r,t)\mathbf{e}_\phi$};

\end{tikzpicture}
\end{center}
\caption{Circularly symmetric flow inside a unit disk.}
\label{fig:circularflow}
\end{figure}

%We consider the anisotropic case, so instead of the term $\nu\Delta$ in (\ref{NSeq}), we use $\nu_1\partial_{rr}+\frac{\nu_1}{r}\partial_r+\frac{\nu_2}{r^2}\partial_{\phi\phi}$, where $\nu_1$ and $\nu_2$ are the viscosity along $\mathbf{e}_r$ and $\mathbf{e}_\phi$. 

Under planar circular symmetry, the Navier-Stokes equations 
  (\ref{NSeq}) become in polar coordinates $(r,\phi)$:
\begin{equation}
    \begin{cases}
    \label{NSECS}
        \frac{-(u_\phi^\nu)^2}{r}+\partial_r p^\nu=0, \quad \quad \quad \quad \quad \quad \quad \quad &\text{in }\Omega\times[0,T],\\
        \partial_tu_\phi^\nu-\frac{\nu_1}{r}\partial_r(r\partial_ru_\phi^\nu)+\nu_2\frac{u_\phi^\nu}{r^2}=f, & \text{in }\Omega\times[0,T],\\
        u_\phi^\nu(0,r)=a(r), & \text{in }\Omega,\\
        u_\phi^\nu(t,r)=0, &\text{for  } r=1,\delta, t\in [0,T],
    \end{cases}
\end{equation}
where we assume again that the pressure is radial and the initial condition is $\bu_0=a(r) \bm{e}_\phi$.
%We observe that the equations for pressure and velocity in (\ref{NSECS}) are decuple and that the system depends from both viscosity coefficients.\\
 In the same way as for the pipe parallel case, we have the following system in Cartesian coordinates $(x,y)$, 
 $r=(x^2+y^2)^{1/2}$, $\tan \phi = y/x$:
\begin{equation}
\label{SemiNSE}
    \begin{cases}
        \partial_t \mathbf{u}_v^\nu=A_\nu \mathbf{u}_v^\nu, \quad \quad \quad & \text{in } \Omega\times[0,T],\\
        \mathbf{u}_v^\nu=(-a(r)\sin\phi, a(r)\cos\phi)=:\mathbf{u}_0, &\text{on } \Omega\times\{t=0\},\\
        \mathbf{u}_v^\nu=0, & \text{in } \partial\Omega\times[0,T],
        \end{cases}
\end{equation}
with $\mathbf{u}_v^\nu=(-u^\nu_\phi(r)\sin\phi, u^\nu_\phi(r)\cos\phi)$, $A_\nu=\nabla \cdot J\nabla$ and 
\begin{equation*}
    J=\left [
\begin{pmatrix}
\nu_1\cos^2\phi+\nu_2\sin^2\phi & (\nu_1-\nu_2)\sin\phi \cos\phi\\
(\nu_1-\nu_2)\sin\phi \cos\phi & \nu_1\sin^2\phi+\nu_2\cos^2\phi
\end{pmatrix}
\right ].
\end{equation*}

%To use the semigroup theory we need that the coefficients of $A_\nu$ are well-defined so for simplicity we consider as domain the annulus: $\Omega=\{x\in \R^2 s.t. \epsilon\leq|x|\leq1\}$ with $\epsilon>0$.
We will study the semigroup generated  by $A_\nu$ primarily using Cartesian coordinates. 
 %We assume that the initial data $\mathbf{u}_0$ is in $L^2(\Omega)$.
 We observe that $A_\nu=\nabla \cdot J\nabla$ is an unbounded operator on $X=L^2(\Omega)$  with dense domain $D(A_\nu)=H^2(\Omega)\cap H^1_0(\Omega)$, which is self-adjoint. In fact, $A_\nu$ is symmetric on $D(A_\nu)$:
\begin{equation*}
    \int_\Omega \left[(\nabla\cdot J\nabla) u\right] v\, dxdy =-\int_\Omega J\nabla u\cdot \nabla v\, dxdy=\int_\Omega u \left[(\nabla \cdot J\nabla ) v\right]\, dxdy,
\end{equation*}
and its range is $X$ by elliptic regularity. In particular, $A_\nu$ is closed.
 Moreover,  $A_\nu$ is a negative operator, since
\begin{equation*}
    (A_\nu u,u)=\int_\Omega \left[(\nabla \cdot J\nabla)u\right] u \, dx dy= - \int_\Omega J\nabla u\cdot \nabla u  \, dxdy\leq -\min(2\nu_1, 2\nu_2)\, ||\nabla u||_{L^2(\Omega)},
\end{equation*}
for  $u\in D(A_\nu)$, recalling that the eigenvalues of $J$ are $2\nu_1$, $2\nu_2$, hence positive. 
As in the case of pipe flows, $A_\nu$ is a (scalar) uniformly strongly elliptic operator. The solvability of the Dirichlet problem for $A_\nu$ by the Lax-Milgram Theorem or the Riesz Representation Theorem  implies by Rellich's Theorem that $A_\nu^{-1}$ is a compact, self-adjoint operator on $L^2(\Omega)$. The Spectral Theorem then implies that the spectrum of $A_\nu$ is real and consists solely of eigenvalues.
Strong ellipticity also implies that  $A_\nu$ is an $m$-dissipative operator on $X$. 
Using a corollary of the Lumer-Phillips Theorem (see e.g. \cite{pazy2012semigroups}), we can conclude that $A_\nu$ is the generator of a strongly continuous or $C_0$ semigroup of contractions $T_\nu(t)=e^{t A_\nu}$, $t\geq 0$. In particular, $T_\nu(t):X\longrightarrow D(A_\nu)$ for $t>0$.
 %The Lumer-Phillips Theorem also imply that $\R^+ \subset \rho(A_\nu)$ and so for every $\lambda >0$, $R(\lambda I-A)=X$.

Furthermore,  the spectrum of $A_\nu$, $\sigma(A_\nu)\subset (-\infty,0)$, while the resolvent set $ \rho(A_\nu) \supset \Sigma_\omega:=\{\lambda \in \mathbb{C}, \arg \lambda \in (-\omega,\omega), \, \forall 0<\omega<\pi\}$. 
Given that $A_\nu$ is self-adjoint, the following estimate holds:
\begin{equation}  \label{eq:ResolventEst}
    \|R(\lambda;A_\nu) \|_{{\mathcal B}(L^2(\Omega))}\leq\frac{C}{|\Im\lambda|}, \qquad \lambda \notin \mathbb{R}
\end{equation}
where, given a Banach space $X$, ${\mathcal B}(X)$ denotes the space of bounded linear operators from $X$ into itself equipped with the standard operator norm $\| \cdot \|_{{\mathcal B}(X)}$ and $R(\lambda;A_\nu)$ is the resolvent operator of $A_\nu$ at $\lambda\in \rho(A_\nu)$. 
Therefore,  $A_\nu$ is sectoral of angle $\omega$ for any $\pi/2<\omega<\pi$, independent of $\nu$, and $T_\nu(t)$ is also an analytic semigroup (see e.g. \cite[Theorem 5.2, page 62]{pazy2012semigroups}).
 We can conclude that, if the initial data $\mathbf{u}_0$ is in $L^2(\Omega)$, the unique solution of problem (\ref{SemiNSE})  in $C(0,T;L^2(\Omega))$ is  given by \ 
 $\mathbf{u}_v^\nu=T_\nu(t) \mathbf{u}_0=e^{t A_\nu }\bu_0$.

%We observe that the Euler limit system of (\ref{SemiNSE}) is given by:
%\begin{equation*}
 %   \begin{cases}
  %      \partial_t\mathbf{u}_v^0=0 \\
   %     \mathbf{u}_v^0|_{t=0}=\mathbf{u}_0
   % \end{cases}
%\end{equation*}
%and so it has the stationary solution $\mathbf{u}_v^0=\mathbf{u}_0$.

\noindent We now turn to the vanishing viscosity limit, establishing the following result.
 
\begin{theorem} \label{t:DiskVL}
 Let  $\mathbf{u}_0 \in L^2(\Omega)$. Then
    \begin{equation} \label{eq:DiskVVL}
        \normone{\mathbf{u}_v^\nu-\mathbf{u}_v^0}=||e^{tA_\nu}\mathbf{u}_0-\mathbf{u}_0||_{L^\infty(0,T;L^2(\Omega))}\to 0
    \end{equation}
    as  $\nu_1, \nu_2 \to 0$, for any fixed $0<T<\infty$.
\end{theorem}

 %We will show that:
%\begin{equation*}
 %   \lim_{\nu_1, \nu_2 \to 0}T_\nu(t)\mathbf{u}_0=\mathbf{u}_0 \quad \text{ in} \quad L^\infty(0;T;L^2(\Omega)).
%\end{equation*}
%It is convenient to separate the two coefficients. To this end, we observe that $J=\nu_1J_1+\nu_2J_2$ with:
%\begin{equation*}
%        J_1=\left [
%\begin{pmatrix}
%\cos^2\phi & \sin\phi \cos\phi\\
%\sin\phi \cos\phi & \sin^2\phi
%\end{pmatrix}\right], \quad \quad         J_2=\left[
%\begin{pmatrix}
%\sin^2\phi & -\sin\phi \cos\phi\\
%-\sin\phi \cos\phi & \cos^2\phi
%\end{pmatrix}
%\right ],
%\end{equation*}
%so that $A_\nu=\nu_1A_1+\nu_2A_2$ with $A_1=(\nabla \cdot J_1\nabla)$ and $A_2=(\nabla\cdot J_2\nabla)$. The two operators do not commute, but it is possible to treat one of the two operators $\nu_i A_i$, $i=1,2$, as a perturbation of the other, depending on the relative size of the viscosity coefficients (see \cite[Theorem 4.5 ]{pazy2012semigroups} for more details).
%To show that:
%\begin{equation*}
 %   \lim_{\nu_1, \nu_2 \to 0 }e^{t(\nu_1A_1+ \nu_2A_2)}\mathbf{u}_0=\mathbf{u}_0 \quad \textit{in} \quad  L^2(\Omega) \quad \textit{unformly in} \quad [0,T] 
%\end{equation*}
%we use results of perturbation analysis 
\begin{proof}
It is enough to show that along any sequence $\nu_n=(\nu_{1,n},\nu_{2,n})$ we have convergence to the initial data, since then the limit is  unique. Given any such sequence (not relabeled), we let $A_n:=A_{\nu_n}$. 
We note that the domain $D:=D(A_n)=H^2(\Omega)\cap H^1_0(\Omega)$ is independent of $n$ and that 
\[
    \lim_{n\to \infty} A_n \bu = A \bu =0, \quad \forall \bu \in D,
\]
where $A$ denotes the zero operator on $X$. $A$ can be viewed as the generator of the trivial group $T_0(t)$, consisting of the identity operator on $X$ for every $t\in\mathbb{R}$.
Since  the semigroup $T_n:=T(A_n)$ is a semigroup of contractions for each $n$, the sequence $A_n$ satisfies the hypotheses of \cite[Theorem 4.5, page 88]{pazy2012semigroups}, a consequence of the Trotter-Kato approximation Theorem. We can therefore conclude that:
\[
    \bu_v^{\nu_n} = T_n \bu_0 \underset{n\to \infty}{\longrightarrow} T_0 \bu_0=\bu_0, \quad \forall \bu_0\in L^2(\Omega), \; t\geq 0,
\]
uniformly in $t$ over bounded intervals. This establishes \eqref{eq:DiskVVL}.
\end{proof}
% Using that for every $\lambda >0$, $(\lambda I-A)D$ is dense in $X$, we have:
%\begin{equation*}
 %   \lim_{\nu_1\to 0}T_\nu(t)\mathbf{u}_0=\lim_{\nu_1\to 0}e^{t\nu_1(A_1+\frac{\nu_2}{\nu_1}A_2)}\mathbf{u}_0=\mathbf{u}_0 
%\end{equation*}
%in $L^2(\Omega)$ and uniformly in $[0,T]$. If $\nu_1\leq \nu_2$ we obtain the same result.

 Using the semigroup approach, it is also possible to estimate the growth of the derivatives near the boundary. 
 As a matter of fact, the operator $A_\nu$ is sectorial with angles that is uniform  w.r.t. $\nu$ and generates an analytic semigroup of contractions. Hence, fractional powers of $-A_\nu$ are defined. In particular, 
 by interpolation $(-A_\nu)^{1/2}$ has domain $H^1_0(\Omega)$, given that $(-A_{\nu})^{1/2} = [I, -A_{\nu}]_{1/2,\infty}$, where $I$ is the identity operator (this result follows for example from \cite[Theorem 6.7.3, page 159]{BerghLofstromBook1976} and also \cite[Theorem 6.10, page 73]{pazy2012semigroups}). 
 
 It will be convenient to work with anisotropic norm. To this end, we observe that, since the bilinear  form associated to  the positive operator $-A_\nu$ is given by
 \begin{equation*}
     \int_\Omega J\nabla u\cdot \nabla v\, dx dy,\qquad u,v\in H^1_0(\Omega),
 \end{equation*}
 standard functional analysis results give:
 \begin{equation} \label{eq:SquareRootNormEquiv}
     \int_\Omega J\nabla u\cdot \nabla u\, dx dy  = \|(-A_\nu)^{1/2} u \|_{L^2(\Omega)}^2
     \qquad u\in H^1_0(\Omega).
 \end{equation}
 Above we used that $-A_\nu$ is self-adjoint, which implies the uniqueness of the self-adjoint square-root operator, which must coincide with $(-A_\nu)^{1/2}$. (We refer the reader to e.g. \cite{Bernau1968} and also the recent article \cite{SebestyenTarcsay2017} for further discussion and relevant references.)
The quadratic form associated to the bilinear form above gives rise to a seminorm in the usual way:
\begin{equation}
    \|| u\|| := \left[ \int_\Omega J\nabla u\cdot \nabla u\, dx dy \right]^{1/2},
 \end{equation}
which is equivalent to the $\dot{H}^1$-seminorm. As a matter of fact, using the orthogonality of the eigenvalues of the matrix $J$, we have that:
\begin{equation} \label{eq:TripleNormEq}
 \|| u\||^2 = \frac{1}{2} \nu^2_1\|\nabla_r u\|^2+ \frac{1}{2}\nu_2^2\|\nabla_\phi u\|^2,
\end{equation}
 where $\nabla_r := {\bf e}_r\cdot \nabla$, $\nabla_\phi := -{\bf e}_\phi\cdot \nabla$.
 By Poincar\'e's inequality it follows that:
\begin{equation} \label{eq:H1NormEquiv}
   c_\Omega  \min(\nu_1,\nu_2)\, \|u\|_{H^1_0(\Omega)} \leq  \|| u\|| \leq C_\Omega  \max(\nu_1,\nu_2)\, \|u\|_{H^1_0(\Omega)},
\end{equation}
with $c_\Omega$, $C_\Omega$ positive constants depending on the domain, but not on $\nu$.

%Next, we distinguish two cases (the case $\nu_1=\nu_2$ is the isotropic case): 
 %\begin{enumerate}[label={\bf (\alph*)}, ref={\bf (\alph*)}]
  %\item $\nu_1<\nu_2$; \label{i:DiskCase1}
   %\item $\nu_1> \nu_2$.  \label{i:DiskCase2}
  %\end{enumerate} 
 %In case \ref{i:DiskCase1}, we write $A_\nu = \nu_1\,A_{\nu'}$, where $\nu'=(1,\nu_2/\nu_1)$ and observe that, by a change of variables, $T_\nu(t)=e^{t A_\nu}= e^{\tau\, A_{\nu'}}$ with $\tau=t\nu_1$. 
% We first observe that   $A_{\nu'}$ is also sectorial with the same angle as $A_\nu$ and $m$-dissipative.
% Furthermore, the constant $C_{\nu'}$ above in \eqref{eq:ResolventEst} applied to $A_{\nu'}$ depends on the ratio between the ellipticity constant  and the operator norm of $A_{\nu'}$(see e.g. the proof of \cite[Theorem 2.7, page 211]{pazy2012semigroups}).
% Since the two eigenvalues of $A_{\nu'}$ are $2$ and $2 \nu_2/\nu_1\geq 2$, $C_\nu'$ can be bounded above by $1$ uniformly in $\nu'$. At the same time, the norm of the operator $A_{\nu'}^{1/2}$ is bounded above by $(\nu_2/\nu_1)^{1/2}$ by interpolation again.
 Then by standard semigroup results  (e.g. \cite[Theorem 6.13, page 69]{pazy2012semigroups}), 
 \begin{equation*}
%\label{tt}
    \| (-A_{\nu})^\frac{1}{2}T_{\nu}(t)\|_{{\mathcal B}(L^2(\Omega))}\leq C_{\Omega} \frac{1}{\sqrt{t}},  \quad t>0,
\end{equation*}
as $\rho(A_\nu)\supset \mathbb{C}\setminus (-\infty,0)$. %\Annacomment{Guardate la dimostrazione del teorema, la costante e' $M_1$  nel caso $0<\alpha<1$ nella notazione del teorema. Ma $M_1=1$  perche' il semigruppo e' di contrazioni e analitico. Questo si vede nella dimostrazione del Teorema 5.2, pagine 61 di Pazy }
%\textcolor{blue}{Valentina. Nell'equazione precedente è $A_{\nu'}$ giusto?}\Annacomment{Si', esatto.}
It follows that %from \eqref{eq:H1NormEquiv} that
 \[
     \|| T_\nu(t) u\|| \leq C_{t,\Omega}\lVert u \rVert_{L^2(\Omega)}  \quad t>0,
 \]
 and by using \eqref{eq:H1NormEquiv} that
 \begin{equation*}
    \|T_\nu u\|_{H^1_0(\Omega)} \leq  \frac{C_{t,\Omega}}{\sqrt{\min(\nu_1,\nu_2)}}\lVert u \rVert_{L^2(\Omega)}.
\end{equation*}
%In case \ref{i:DiskCase2}, we simply switch the role of $\nu_1$ and $\nu_2$ and conclude that
%\[
%  \| \nabla T_\nu(t) \|_{{\mathcal B}(L^2(\Omega))} \leq \frac{C_t}{\sqrt{\nu_2}}, \quad t>0.
%\]
These rates are consistent with those for the heat equation if $\nu_1=\nu_2$, for which they are sharp.
% We observe that $(-A_\nu)^\frac{1}{2}$ is a well-defined self-adjoint operator with domain $D((-A_\nu)^\frac{1}{2})=H^1_0(\Omega)$. Then we infer:
%We observe that if $u \in D(A_\nu)$, then:
%\begin{equation*}
%\begin{aligned}
 %   &||(-A_\nu)^\frac{1}{2}u||_{L^2(\Omega)}^2=((-A_\nu)^\frac{1}{2}u, (-A_\nu)^\frac{1}{2}u)=-(A_\nu u,u)\\
  %  &-\int_\Omega(\nabla\cdot J\nabla)uu=\int_\Omega(J\nabla u \cdot \nabla u)\geq \min(2\nu_1, 2\nu_2)||\nabla u||_{L^2(\Omega)}^2.
   % \end{aligned}
%\end{equation*}
%Using (\ref{tt}), we can consider two cases:
%\begin{enumerate}
%    \item If $\nu_1\leq \nu_2$: $||\nabla T_\nu(t)||\leq \frac{C_t}{\sqrt{\nu_1}}$ for $t>0$,
 %   \item If $\nu_2 <\nu_1$: $||\nabla T_\nu(t)||\leq \frac{C_t}{\sqrt{\nu_2}}$ for $t>0$.
%\end{enumerate}
If we assume more regularity on the initial data, i.e.,  $\mathbf{u}_0 \in H^1(\Omega)$, we obtain also that for $t>0$,
\begin{equation*}
    ||\nabla T_\nu (t)\mathbf{u}_0-\nabla \mathbf{u}_0||_{L^2(\Omega)}\leq \| \nabla T_\nu(t)\|_{{\mathcal B}(L^2(\Omega))}||\mathbf{u}_0||_{L^2(\Omega)}+||\mathbf{u}_0||_{H^1(\Omega)}\leq \frac{C_{t,\Omega}}{\sqrt{\min(\nu_1,\nu_2)}}\lVert \mathbf{ u}_0 \rVert_{H^1(\Omega)}.
\end{equation*}
%\Annacomment{Secondo me qui ci va $\dfrac{\max(\nu_1,\nu_2)}{\min(\nu_1,\nu_2)^2}^{1/2}$.}
In other words, the $H^1$-norm  diverges as $\nu\to 0$ with a maximal growth rate proportional to $\frac{1}{\sqrt{\min(\nu_1,\nu_2)}}$.

%t could be possible to extend all these results also near the origin considering viscosity coefficients that become constant and isotropic near $r=0$.

 To conclude the analysis it would be interesting also to study the behavior of the vorticity. However, if $\nu_1 \not= \nu_2$, the matrix $J$ is not constant and the vorticity equations are not easy to determine and to study.\vspace{1em}

The main contribution of this work is a first analysis of the vanishing viscosity limit for symmetric flows with anisotropic viscosity, a setting that had not yet been studied in previous works.
\noindent We proved that in all three symmetric configurations the vanishing viscosity limit can be established in the anisotropic setting. An important observation is that the rate of convergence depends on both viscosity coefficients, with the coefficients normal to the boundary being the most relevant parameter.
\noindent This work opens several directions for further investigation. In particular, a more detailed analysis of the limit in curved domains, as well as an extension of the semigroup approach to the study of derivative estimates in the three-dimensional settings, represent a natural continuation.

\subsection*{Acknowledgements} 
%\addcontentsline{toc}{section}{Ackowledgement}
%\thispagestyle{empty} % Remove all styles for this page to keep it plain
\noindent This work is partially supported by the US National Science Foundation grant DMS-2206453 under the supervision of Anna Mazzucato and by the ERC STARTING GRANT 2021 “Hamiltonian Dynamics, Normal
Forms and Water Waves” (HamDyWWa), Project Number: 101039762 under the supervision of Riccardo Montalto.  The Views and opinions expressed are however those of the authors only and do not necessarily reflect those of the European Union or the European Research Council. Neither the European Union nor the granting authority can be held responsible for them.
V. Galbiati thanks the Mathematics Department at Penn State University, while A. Mazzucato thanks the Mathematics Department at Milan University, for their hospitality. V. Galbiati is also founded by the University of Milan under the mobility program Thesis Abroad call 2024/2025 (I ed.).

\appendix
\section{Corrector estimates} \label{s:rates}
%metti tutti lemmi e robe varie che usi (tipo ASL)
In this appendix,   we discuss the estimates for the correctors that are used throughout the paper. They can be obtained in a way similar to those in the the isotropic case \cite{GIE20191237,plane}. Therefore we omit some details in the proofs.

 We begin by studying the simpler setting of parallel channel flows. We treat only one of the two correctors since the bounds for the other correctors are the same.  We recall that the corrector $\theta_1^0$,  solves the 1D system on $[0,\infty)$:
\begin{equation}
    \begin{cases}
    \label{heat}
        &\partial_t\theta_1^0-\partial_{ZZ}\theta_1^0=0,\\
        &\theta_1^0|_{Z=0}=-u_1^0(t,0), \quad  \theta_1^0|_{Z=\infty}=0,\\
        &\theta_1^0|_{t=0}=0.
    \end{cases}
\end{equation}
Our first result is the following theorem.

\begin{theorem}
\label{teo corr}
Let $u_1^0\in L^\infty(0,T,H^1(0,1))$. Then for any $l\in \mathbb{Z_+}$ there exist constants $C_1>0$ and $C_2>0$, depending on $T$, $l$, and on $||u_1^0(t,z)||_{L^\infty(0,T,H^1(0,1))}$, such that:
\begin{align}
    ||\langle Z \rangle^l\theta_1^0||_{L^\infty(0,T; L^2(0,+\infty))}+||\langle Z \rangle^l\partial_Z\theta_1^0||_{L^2((0,T)\times(0,+\infty))}\leq C_1, \nonumber \\
    ||\langle Z \rangle^l\theta_1^0||_{L^\infty((0,T)\times (0, + \infty))}\leq C_2. \label{eq:teo corr}
\end{align}
\end{theorem}

\begin{proof}
We define $w(t,Z):=\theta_1^0(t,Z)+u_1^0(t,0)e^{-Z}$ and we observe that we have imposed enough regularity on $u_1^0$ have a well-defined trace at $z=0$. Then $w$ satisfies:
\begin{equation}
    \begin{cases}
    \label{heat2}
        &\partial_t w-\partial_{ZZ}w=u_1^0(t,0)e^{-Z},\\
        &w|_{Z=0}= w|_{Z=\infty}=0,\\
        &w|_{t=0}=0.
    \end{cases}
\end{equation}
 Multiplying (\ref{heat2}) by $\langle Z \rangle ^{2l}w$ and integrating over $(0,+\infty)$ gives:
\begin{equation*}
    \begin{aligned}
        &\frac{1}{2}\frac{d}{dt}||\langle Z \rangle^lw||^2_{L^2(0,+\infty)}+\int_0^{+\infty}\partial_Zw\partial_Zw\langle Z \rangle^{2l}+\int_0^{+\infty}l(1+Z^2)^{l-1}2Zw\partial_Zw\\
        &\qquad \qquad \qquad \leq |u_1^0(t,0)|\, ||e^{-Z}\langle Z \rangle^{l}||_{L^2(0,+\infty)}\, ||\langle Z \rangle^{l}w||_{L^2(0,+\infty)}.
    \end{aligned}
\end{equation*}
We bound the last term on the left-hand side as
\begin{equation*}
\begin{aligned}
    &\int_0^{+\infty}l(1+Z^2)^{l-1}2Zw\partial_Zw\leq c_l\int_0^{+\infty}(1+Z^2)^lw\partial_Zw\leq c_l||\langle Z \rangle^lw||_{L^2(0,+\infty)}||\langle Z \rangle^l\partial_Zw||_{L^2(0,+\infty)}\\
    &\qquad \qquad \leq \frac{c_l^2}{2}||\langle Z \rangle^lw||_{L^2(0,+\infty)}^2+\frac{1}{2}||\langle Z \rangle^l\partial_Zw||_{L^2(0,+\infty)}^2,
    \end{aligned}
\end{equation*}
for some constant $c_l$.
By integrating over time and using Grönwall's inequality, it follows that
\begin{equation*}
    ||\langle Z \rangle^lw||^2_{L^{\infty}(0,T;L^2 (0, + \infty))}+||\langle Z \rangle^l\partial_Zw||_{L^2((0,T)\times  (0, + \infty))}^2\leq C(l,T,||u_1^0(\cdot,0)||_{L^\infty(0,T)}).
\end{equation*}

Thus the following estimates for $\theta_1^0$ holds:
\begin{equation*}
\begin{aligned}
    &||\langle Z \rangle^l\theta_1^0||_{L^\infty(0,T; L^2 (0, + \infty))}+ ||\langle Z \rangle^l\partial_Z\theta_1^0||_{L^2((0,T)\times  (0, + \infty))}\leq ||\langle Z \rangle^lw||_{L^\infty(0,T;L^2 (0, + \infty))}\\
    & \qquad + ||\langle Z \rangle^l\partial_Zw||_{L^2((0,T)\times  (0, + \infty))}+||\langle Z \rangle^l u_1^0(t,0)e^{-Z}||_{L^\infty(0,T;L^2 (0, + \infty))}\\
    &\qquad \qquad + ||\langle Z \rangle^l\partial_Zu_1^0(t,0)e^{-Z}||_{L^2((0,T)\times  (0, + \infty))}
    \leq C_1 (||u_1^0(t,z)||_{L^\infty(0,T,H^1(0,1))}),
    \end{aligned}
\end{equation*}
using that $\langle Z \rangle e^{-Z}$ is uniformly bounded on $[0,\infty)$.

Next,  we observe that by the maximum principle for the heat equation, we have: 
\[
  ||\theta_1^0||_{L^\infty(0,T;L^\infty  (0, + \infty)))}\leq G:=||u_1^0||_{L^\infty((0,T)\times(0,1))}.
\]
We let $k(t,Z):=Ge^{t-Z}$, which  satisfies the following problem:
\begin{equation*}
    \begin{cases}
        &\partial_tk-\partial_{ZZ}k=0\\
        &k|_{t=0}=Ge^{-Z}\geq 0 \quad k|_{Z=0}=Ge^{t}>G \quad k|_{Z=\infty}=0.
    \end{cases}
\end{equation*}
By the maximum principle for the heat equation, it follows that $\theta_1^0\leq Ge^{T-Z}$. Then for any $l\geq 0$, 
\begin{equation*}
    ||\langle Z \rangle ^l\theta_1^0||_{L^\infty((0,T) \times (0, + \infty))}\leq G||\langle Z \rangle ^le^{T-Z}||_{L^\infty(0,+\infty)}\leq G e^T,
\end{equation*}
using that $\langle Z \rangle^l e^{-Z}$ is uniformly bounded in $Z$ for any $l$.
\end{proof}
 To tackle the behavior of the vorticity in the zero-viscosity limit, we need higher-order estimates.

\begin{theorem}
    Under the hypotheses of Theorem \ref{teo corr}, for $1< p\leq\infty$:
    \begin{equation*}
        \begin{aligned}
        &||\theta_1^0||_{L^\infty(0,T;L^p(0,1))}\leq \nu_3^\frac{1}{2p}C(T,||u_1^0||_{L^\infty(0,T;H^1(0,1))}),\\
    &||\partial_z\theta_1^0||_{L^\infty(0,T;L^2(0,1))}\leq \nu_3^{-\frac{1}{4}}C(T,||u_1^0||_{L^\infty(0,T;H^1(0,1))}),\\
    &||\partial_z\theta_1^0||_{L^\infty(0,T;L^1(0,1))}\leq C(T,||u_1^0||_{L^\infty(0,T;H^1(0,1))}),
        \end{aligned}
    \end{equation*}
    where in this case we are considering the integrals respect to $z$, not to $Z$.
    
\end{theorem}

\begin{proof}
    We consider again the system  (\ref{heat2}) for $w$.
    Multiplying (\ref{heat2}) by $w^{p-1}$ for $p>1$ and integrating over $(0,+\infty)$ yields:
    \begin{equation*}
    \begin{aligned}
        &\frac{1}{p}\frac{d}{dt}||w||_{L^p(0,+\infty)}^p+(p-1)\int_0^\infty\partial_Zww^{p-2}=\int_0^\infty u_1^0(t,0)e^{-Z}w^{p-1}\\
        &\qquad \leq\left(\int_0^\infty u_1^0(t,0)^pe^{-pZ}\right)^\frac{1}{p}\left(\int_0^\infty w^p\right)^\frac{p-1}{p}=||u_1^0(t,0)e^{-Z}||_{L^p(0,\infty)}||w||_{L^p(0,\infty)}^{p-1}\\
        &\qquad \qquad \leq c_1||e^{-Z}||_{L^p(0,+\infty)}^p+c_2||w||_{L^p(0,\infty)}^p\leq c_1+c_2||w||_{L^p(0,\infty)}^p.
        \end{aligned}
    \end{equation*}
    We then integrate in time and apply Grönwall's inequality to obtain:
    \[
        ||w||_{L^\infty(0,T,L^p(0,+\infty))}\leq C, 
    \]    
    which implies
    \[
       ||\tilde{w}||_{L^\infty(0,T,L^p(0,1))}\leq \nu_3^\frac{1}{2p}C(T,||u_1^0(\cdot,0)||_{L^\infty(0,T)}),
    \]
    where $\tilde{w}(t,z)= \varphi(z) \, w(t,Z)$ with $\varphi$ a suitable cut-off.
    We note that $\partial_z w$ satisfies
    \begin{equation}
    \label{w}
        \frac{\partial}{\partial t}\frac{\partial w}{\partial Z}-\partial_{ZZ}\frac{\partial w}{\partial Z}=-u_1^0(t,0)e^{-Z},
    \end{equation}
    with the conditions:
    \begin{equation*}
    \begin{aligned}
        &\partial_{ZZ}w=-u(t,0) \quad \textit{at }\quad  z=0,\\
        &\partial_Zw \to 0 \quad \textit{as } \quad z\to \infty ,\quad \partial_Zw=0 \quad \textit{at } \quad t=0.
        \end{aligned}
    \end{equation*}
    Multiplying (\ref{w}) by $\partial_Zw$ and integrating over $(0,\infty)$ gives:
    \begin{equation*}
        \begin{aligned}
            &\frac{1}{2}\frac{d}{dt}||\partial_Zw||_{L^2(0,+ \infty)}^2+||\partial_{ZZ}w||_{L^2(0,+ \infty)}^2\leq c||\partial_Zw||_{L^2(0,+ \infty)}+|\partial_{ZZ}w\partial_Zw|_{Z=0}\\
            &=c||\partial_Zw||_{L^2(0, + \infty)}+|u_1^0(t,0)||\partial_Zw|_{Z=0}\leq c^2+||\partial_Zw||_{L^2(0,+ \infty)}^2.
        \end{aligned}
    \end{equation*}
   We once again integrate in time and apply Grönwall's inequality:
    \begin{equation*}
        ||\partial_Zw||_{L^\infty(0,T;L^2(0,+ \infty))}\leq C\Rightarrow||\partial_zw||_{L^\infty(0,T;L^2(0,1))}\leq \nu_3^{-\frac{1}{4}}C(T,||u_1^0(\cdot,0)||_{L^\infty(0,T)}).
    \end{equation*}
    To treat the case $p-1$,  we introduce a standard regularization of the absolute value: $F_\lambda(x)=\sqrt{\lambda^2+x^2}$, and use it  too approximate $L^1$-norm:
    \begin{equation}
    \label{l1}
    \begin{aligned}
        &\frac{d}{dt}\int_0^\infty F_\lambda(\partial_Zw)dZ=\int_0^\infty F'_\lambda(\partial_Zw)\frac{\partial}{\partial t}(\partial_Zw)dZ\\
        &\qquad \qquad =-\int_0^\infty F'_\lambda(\partial_Zw)\partial_{ZZ}w-\int_0^\infty F'_\lambda(\partial_Zw)u_1^0(t,0)e^{-Z}dZ.
        \end{aligned}
    \end{equation}
    From the convexity of $F_\lambda$:
    \begin{equation*}
        \partial_{ZZ}(F_\lambda(\partial_Zw))=F'_\lambda(\partial_Zw)\partial_{ZZ}w+F''_\lambda(\partial_Zw)(\partial_{ZZ}w)^2\geq F'_\lambda(\partial_Zw).
    \end{equation*}
    Then integrating by parts the first term on the right-hand-side of (\ref{l1}) gives
    \begin{equation*}
        \int_0^\infty F'_\lambda(\partial_Zw)\partial_{ZZ}w\leq \int_0^\infty \partial_{ZZ}(F_\lambda(\partial_Zw))dZ=F'_\lambda(\partial_Zw)\partial_{ZZ}w|_{Z=0},
    \end{equation*}
    so that  (\ref{l1}) becomes:
    \begin{equation*}
        \frac{d}{dt}\int_0^\infty F_\lambda(\partial_Zw)dZ\leq F'_\lambda(\partial_Zw)\partial_{ZZ}w|_{Z=0}+\int_0^\infty F'_\lambda(\partial_Zw)u_1^0(t,0)e^{-Z}dZ.
    \end{equation*}
    Considering that $|F'_\lambda(x)|\leq 1$ and that the others terms are bounded, we obtain:
    \begin{equation*}
        \frac{d}{dt}\int_0^\infty F_\lambda(\partial_Zw)dZ\leq c,
    \end{equation*}
    with $c$ independent of $\lambda$.
    We again integrate in time and  take the limit for $\lambda\to 0$:
        \begin{equation*}
        \lim_{\lambda\to 0}\int_0^\infty F_\lambda(\partial_Zw)dZ\leq c,
    \end{equation*}
    uniformly in $\nu>0$ for $T$.
    Finally,the Monotone Convergence Theorem implies
    \begin{equation*}
        ||\partial_Zw||_{L^\infty(0,T;
        L^1(0,\infty))}\leq C \Rightarrow||\partial_zw||_{L^\infty(0,T;L^1(0,1))}\leq C(T,||u_1^0(\cdot,0)||_{L^\infty(0,T)})
    \end{equation*}
    from which the results for $\theta_1^0$ follows.
\end{proof}

 We now turn to  $\theta_2^0$, which satisfies the following system:
\begin{equation}
    \begin{cases}
    \label{theta2}
        &\partial_t\theta_2^0+\theta_1^0\partial_x\theta_2^0+u_1^0(t,0)\partial_x\theta_2^0-\partial_{ZZ}\theta_2^0=-\theta_1^0\partial_xu_2^0(t,x,0)\\
        &\theta_2^0|_{Z=0}=-u_2^0(t,x,0), \theta_2^0|_{Z=\infty}=0\\
        &\theta_2^0|_{t=0}=0.
    \end{cases}
\end{equation}
We collect needed estimates for this corrector component in the following result, recalling that his spatial domain is given by $\Omega_\infty=[0,L]\times[0,\infty)$:

\begin{theorem} \label{teo corr2}
Let $u_1^0\in L^\infty(0,T;H^1(0,1))$ and $u_2^0\in L^\infty(0,T;H^2(\Omega))$. Then, for each $l \in \mathbb{Z_+}$, there exists $C_1>0$, $C_2>0$, depending on $||u_1^0||_{L^\infty((0,T)\times(0,1))}, ||u_2^0||_{L^\infty((0,T);H^2(\Omega))}$  such that for $i=0,1,2,3,4$:
\begin{align} 
    &||\langle Z \rangle^l\partial_x^i\theta_2^0||_{L^\infty(0,T;L^2(\Omega_\infty))}+||\langle Z \rangle ^l\partial_x^i\partial_Z\theta_2^0||_{L^2(0,T;L^2(\Omega_\infty))}\leq C_1, \nonumber \\
    &||\langle Z \rangle^l\partial_x^i\theta_2^0||_{L^\infty((0,T)\times \Omega_\infty)}\leq C_2.
    \label{eq:teo corr2}
\end{align}
\end{theorem}

\begin{proof} With slight abuse of notation, we let $w(t,Z):=\theta_2^0(t,x,Z)+e^{-Z}u_2^0(t,x,0)$ and we observe that we have enough regularity on the initial condition such that $u_2^0$ is a classical solution of (\ref{EEPP}), so we can restrict this equation to the boundary $z=0$, in order to obtain an equation satisfies by $u_2^0(t,x,0)$.
The function $w$ satisfies the following system on $\Omega_\infty$:
\begin{equation}
    \begin{cases}
    \label{omega}
        &\partial_tw+\theta_1^0\partial_xw+u_1^0(t,0)\partial_xw-\partial_{ZZ}w=-\theta_1^0\partial_xu_2^0(t,\cdot,0)+H,\\
        &w|_{Z=0}=0, w|_{Z=\infty}=0,\\
        &w|_{t=0}=0,
    \end{cases}
\end{equation}
where $H(t,x,Z):=e^{-Z}(\theta_1^0(t,Z)\partial_xu_2^0(t,x,0)+u_2^0(t,x,0))$.
 Multiplying the first equation in (\ref{omega}) by $\langle Z \rangle ^{2l}w$ and integrating by parts over $\Omega_\infty$ gives:
\begin{equation*}
\begin{aligned}
    &\frac{1}{2}||\langle Z \rangle ^{l}w||_{L^2(\Omega_\infty)}^2+||\langle Z \rangle ^{l}\partial_Zw||_{L^2(\Omega_\infty)}^2\leq ||u_2^0(t,x,0)||_{H^1(0,L)}||\langle Z \rangle ^l\theta_1^0||_{L^2(\Omega_\infty)}||\langle Z \rangle ^l w||_{L^2(\Omega_ \infty)}\\
    &+||u_2^0(t,x,0)||_{H^2(0,L)}||\langle Z \rangle^l e^{-Z}||_{L^2(\Omega_\infty)}||\langle Z \rangle ^lw||_{L^2(\Omega_\infty)}+\frac{c_l^2}{2}||\langle Z \rangle ^{l}w||_{L^2(\Omega_\infty)}^2+\frac{1}{2}||\langle Z \rangle ^{l}\partial_Zw||_{L^2(\Omega_ \infty)}^2.
    \end{aligned}
\end{equation*}
We next use results in \autoref{teo corr} and apply Grönwall's inequality to conclude that:
\begin{equation*}
\begin{aligned}
    &||\langle Z \rangle ^{l}w||_{L^\infty(0,T; L^2(\Omega_\infty))}+||\langle Z \rangle ^{l}\partial_Zw||_{L^\infty(0,T;L^2(\Omega_ \infty))}\\
    &\qquad \qquad \leq C\Big(||u_1^0||_{L^\infty((0,T)\times (0, + \infty))}, ||u_2^0||_{L^\infty((0,T);H^2(\Omega))},T,l \Big).
    \end{aligned}
\end{equation*}
Since the derivatives $\partial_x^i\theta_2^0$ $i=1,2,3,4$ satisfy the same equation as $\theta_2^0$, similar estimates holds.
Similarly, by multiplying both sides of (\ref{omega}) by $p|w|^{p-2}w, p>2$, and integrating by parts, it follows that
\begin{equation*}
\begin{aligned}
    &\frac{d}{dt}||w||^p_{L^p(\Omega_\infty)}+\frac{4(p-1)}{p}||\partial_Z|w|^\frac{p}{2}||_{L^2(\Omega_  \infty)}^2\\
    &\qquad \qquad \leq p||u_2^0(t,.,0)||_{H^2(0,L)}||\theta_1^0||_{L^p(0,+\infty)}||w||^{p-1}_{L^p(\Omega)}+p||H||_{L^p(\Omega_\infty)}||w||_{L^p(\Omega_ \infty)}^{p-1}.
    \end{aligned}
\end{equation*}
We then divide both sides by $||w||_{L^p(\Omega_\infty)}^{p-1}$:
\begin{equation*}
    \frac{d}{dt}||w||_{L^p(\Omega_ \infty)}\leq ||u_2^0(t,.,0)||_{L^\infty(0,L)}||\theta_1^0||_{L^p(0,+ \infty)}+||H||_{L^p(\Omega_ \infty)}.
\end{equation*}
Finally, by taking the limit $p\to +\infty$ and integrating in time we have:
\begin{equation*}
||w||_{L^\infty((0,T)\times \Omega_\infty )}\leq C\Big(||u_1^0||_{L^\infty((0,T)\times (0, + \infty))}, ||u_2^0||_{L^\infty(0,T;H^2(\Omega))},T,l \Big).
\end{equation*}
\end{proof}

 By proceeding similarly to what done for $\theta_1^0$, we can establish higher-order estimates for $\theta_2^0$.
 
\begin{theorem} \label{teo corr3}
   Under the hypotheses of Theorem \ref{teo corr2} for $1< p\leq\infty$:
    \begin{equation} 
        \begin{aligned}
        &||\theta_2^0||_{L^\infty(0,T;L^p(0,1))}\leq \nu_3^\frac{1}{2p}C(||u_1^0||_{L^\infty((0,T)\times (0, + \infty))}, ||u_2^0||_{L^\infty((0,T);H^2(\Omega))},T),\\
    &||\partial_z\theta_2^0||_{L^\infty(0,T;L^2(0,1))}\leq \nu_3^{-\frac{1}{4}}C(||u_1^0||_{L^\infty((0,T)\times (0, + \infty))}, ||u_2^0||_{L^\infty((0,T);H^2(\Omega))},T),\\
    &||\partial_z\theta_2^0||_{L^\infty(0,T;L^1(0,1))}\leq C\Big(||u_1^0||_{L^\infty((0,T)\times (0, + \infty))}, ||u_2^0||_{L^\infty((0,T);H^2(\Omega))},T \Big),
        \end{aligned}
        \label{eq:teo corr3}
    \end{equation}
    where in this case we are considering the integrals respect to $z$ instead of $Z$.
\end{theorem}

Using the fact that the radial coordinate $r$ is bounded above and below in annular domains, we can establish equivalent estimates for the correctors in the case of pipe parallel flows. For brevity, we omit the details.

We conclude this appendix recalling the statement of the Anisotropic Sobolev Lemma, which is  used throughout the paper.

\begin{theorem}
    \label{th:ASL}
    For all $u \in H^1_0(\Omega)$, it holds that
    \begin{equation*}
        \lVert u \rVert_{L^\infty(\Omega)}\leq C(\lVert u\rVert_{L^2(\Omega)}^\frac{1}{2}\lVert \partial_z u\rVert_{L^2(\Omega)}^\frac{1}{2}+\lVert \partial_z u\rVert_{L^2(\Omega)}^\frac{1}{2}\lVert \partial_x u\rVert_{L^2(\Omega)}^\frac{1}{2}+\lVert u\rVert_{L^2(\Omega)}^\frac{1}{2}\lVert \partial_x\partial_z u\rVert_{L^2(\Omega)}^\frac{1}{2}).
    \end{equation*}
\end{theorem}

\vspace{0.2in}

\printbibliography

\vspace{0.2in}

\end{document}